\documentclass[11pt,a4paper]{article}
\usepackage[margin=1in]{geometry}
\usepackage{amsmath,amssymb,amsthm,mathtools}
\usepackage{mathrsfs}
\usepackage{bm}
\usepackage{hyperref}
\hypersetup{colorlinks=true,linkcolor=blue,citecolor=blue,urlcolor=blue}
\usepackage{xcolor}
\usepackage[normalem]{ulem}
\usepackage[T1]{fontenc}
\usepackage{lmodern}
\usepackage{microtype}

\newtheorem{theorem}{Theorem}
\newtheorem{lemma}{Lemma}

\newtheorem{definition}{Definition}
\newtheorem{assumption}{Assumption}
\newtheorem{remark}{Remark}
\newtheorem{example}{Example}

\numberwithin{equation}{section}

\newcommand{\R}{\mathbb{R}}

\newcommand{\X}{X}

\newcommand{\ip}[2]{\langle #1,#2\rangle}
\newcommand{\norm}[1]{\lVert #1\rVert}

\newcommand{\RR}{\ensuremath{\mathbb{R}}}

\newcommand{\calZ}{\ensuremath{\mathcal{Z}}}
\newcommand{\calS}{\ensuremath{\mathcal{S}}}

\newcommand{\calE}{{\mathcal E}}

\DeclareMathOperator{\supp}{spt}

\DeclareMathOperator{\interior}{int}

\title{Mollified Christoffel-Darboux Kernels and Density Recovery on Varieties\footnote{This work was partially supported by the French-Uruguayan Laboratory of Mathematics IFUMI-CNRS IRL 2030. M. Velasco was partially supported by Fondo Clemente Estable grant FCE-1-2023-1-176172 (ANII). L. Bentancur acknowledges support from a PhD fellowship from the Comisión Académica de Posgrado, Universidad de la República, Uruguay.}}
\author{
Leandro Bentancur\footnote{Facultad de Ciencias, Universidad de la República, Montevideo, Uruguay. Email: leandrob@cmat.edu.uy}
\and
Didier Henrion\footnote{LAAS-CNRS, University of Toulouse, France and Faculty of Electrical Engineering, Czech Technical University in Prague, Czechia. Email: henrion@laas.fr}
\and
Mauricio Velasco\footnote{Facultad de Ciencias, Universidad de la República, Montevideo, Uruguay. Email: mvelasco@cmat.edu.uy}
}

\date{\today}

\begin{document}
	\maketitle
	
    \begin{abstract}
        We introduce mollified Christoffel-Darboux (CD) kernels on varieties, a systematic regularization of the classical CD kernel associated with a probability measure on a compact domain. The main motivations are twofold: first, to sharpen the classical on/off-support dichotomy of the CD polynomial by replacing linear growth on the support by a uniform bound; second, to obtain consistent and quantitatively controlled recovery of densities from moment data, without the need to know the equilibrium measure of the underlying domain.
        
        Our contributions are the following: (i) We introduce families of mollifiers on algebraic varieties. For each measure and degree on such a variety we define a mollified CD kernel, which can be computed from the moments of the underlying measure by linear algebra. (ii) We prove, by elementary arguments, that an improved dichotomy property holds: on the interior of the support the mollified CD polynomial is uniformly bounded in the degree, while outside the support it grows exponentially with the degree. (iii) Assuming Sobolev regularity of the density with respect to a reference measure, we derive explicit convergence rates for density recovery for measures in Euclidean space via mollified CD kernel. (iv) On the unit sphere, we show that suitably chosen algebraic mollifiers, constructed from zonal polynomials, lead to kernels with improved rates, building on classical constructive approximation results.
    \end{abstract}

    

	\section{Introduction}
	
	Let $X\subset\R^n$ be a compact set with nonempty interior, and let $\mu$ be a Borel probability measure supported on $\X$. The \emph{Christoffel-Darboux (CD) kernel} associated with $\mu$ and a finite-dimensional subspace $V_d\subset {\mathscr L}^2(\mu)$ (typically the space of polynomials of degree at most $d$) plays a central role in approximation theory, potential theory and numerical analysis. In particular, the diagonal of the kernel - the \emph{CD polynomial} or its reciprocal the \emph{Christoffel function} - encodes fine geometric information about the support of $\mu$ and have found many applications in convex optimization and data science. For example, in the context of the moment-SOS hierarchy, the CD polynomial is used to
    quantify the convergence rates for solving non-convex polynomial optimization problems \cite{Slot2022}, or to approximate the transport map when solving optimal transport problems \cite{MulaNouy2024}. It is a powerful tool in data analysis and machine learning \cite{Fanuel2022,RoosSlot2023,Adcock2023}. The Christoffel function is fundamental in approximation theory when studying least squares sampling \cite{DolbeaultCohen2021,Adcock2025}.
    For an overview see \cite{LasserrePauwelsPutinar2022} and references therein.
    
	The classical CD polynomial enjoys a striking \emph{dichotomy property} used to recover the support of the measure $\mu$ from moments. Denote by $n_d=\dim V_d$ the dimension of the finite-dimensional space. Under mild regularity assumptions on $\mu$ and $\X$, one has, as $d\to\infty$:
	\begin{itemize}
		\item On the interior of the support of $\mu$, the CD polynomial  grows linearly in $n_d$.
		\item Outside the support of $\mu$, the CD polynomial grows exponentially fast in $d$, with a growth rate that depends on the distance to the support of $\mu$.
	\end{itemize}
	The CD polynomial is the squared norm in $V_d$ of the \emph{evaluation functional} at $x$. As the dimension $n_d$ grows, point evaluation becomes increasingly expensive in norm on the support, while outside the support the norm must blow up exponentially.
	
	When $\mu$ admits a density $f$ with respect to a reference measure $\lambda$ (for instance the Lebesgue measure on $\X$ with suitable regularity assumptions), the classical CD polynomial does not directly recover $f$. Rather, under appropriate assumptions and scaling, it converges to $f$ multiplied by the density of the \emph{equilibrium measure} of the domain $\X$ (this measure is intrinsic, determined by the pluripotential theory of the set $X$). In particular, density recovery from the classical CD kernel requires explicit knowledge of the equilibrium measure, which is available only in special cases (such as boxes, balls, or other highly symmetric domains).
	
	A modified Christoffel function was recently proposed by Lasserre~\cite{Lasserre2023} to address this limitation. The key idea is to relax the point-evaluation constraint in the variational characterization of the Christoffel function and replace it by a weaker constraint that yields a bounded ${\mathscr L}^2$ norm of the corresponding linear functional as the dimension $n_d$ grows. This yields an estimator that converges directly to the density  without requiring the knowledge of the equilibrium measure of the support.
	
	The purpose of the present paper is to generalize and systematize this idea along three axes:
	\begin{enumerate}
		\item We introduce the notion of  {variety with mollifiers} and the ensuing \emph{mollified Christoffel-Darboux (MCD) kernels}, defined by replacing the point evaluation with the smoothing operators associated to a given family of mollifiers at a smoothing scale $\epsilon$. This yields a family of regularized CD kernels parametrized by the choice of mollifier. Our definition contains as a special case the modified CD kernel of \cite{Lasserre2023}.

		\item We prove an improved dichotomy property: for fixed $\varepsilon>0$, the  MCD  polynomial is uniformly bounded on the interior of the support of $\mu$ as $d\to\infty$, while it grows exponentially fast off the support. The proof of exponential growth uses an elementary polynomial of tensor-product type, much simpler than the  Kro\'o-Lubinsky needle polynomial based on Chebyshev polynomials used in earlier works \cite[Section 4.4]{LasserrePauwelsPutinar2022}.
		\item Assuming Sobolev regularity of the density $f$ we give several quantitative recovery guarantees obtained by balancing a \emph{projection error} (due to finite-dimensional approximation) and an \emph{approximation error} (due to mollification). More specifically, (i) We derive explicit convergence rates for density recovery from the MCD kernel for measures supported on a regular compact domain in $\RR^n$ $(ii)$ When $\X$ is the sphere, we construct \emph{algebraic mollifiers} from zonal polynomials and use Ragozin's constructive approximation results~\cite{Ragozin1971} to obtain rates that strictly improve on the current estimates known for spheres~\cite[Corollary 5.3.5]{LasserrePauwelsPutinar2022}.
	\end{enumerate}

	The paper is organized as follows: In Section~\ref{sec:MCDKs} we introduce the MCD kernel on varieties and discuss its basic properties. In Section~\ref{sec: SL_and_DE} we introduce our two main uses for MCD kernels: as support locators and as density estimators, we prove the improved dichotomy property and we establish the desired asymptotic behavior. In Section~\ref{sec: Quant} we derive density recovery results and convergence rates. More specifically, Section~\ref{sec: QuantRn} focuses on quantitative rate estimates for density recovery in $\RR^n$ whereas Section~\ref{sec:sphere} is devoted to the spherical case, where algebraic mollifiers lead to improved approximation rates. Finally, Section~\ref{sec:numerics} contains a numerical illustration of our methods.

\section{Mollified Christoffel-Darboux kernels}
\label{sec:MCDKs}

In this Section we introduce MCD kernels on varieties. By an algebraic variety in $\RR^n$ we mean any set consisting of the common real zeroes of a finite collection of polynomials in $x_1,\dots, x_n$ with real coefficients.

\begin{definition} \label{def: space} A real algebraic variety $\calZ\subseteq \RR^n$ with a reference Borel measure $\lambda$ is endowed with ${\mathscr L}^2$-mollifiers if there exists a set $\calE\subseteq (0,\infty)$ with $0\in \overline{\calE}$ and a collection of functions $\phi_{z,\epsilon}: \calZ\rightarrow \RR$ for $z\in \calZ$, $\epsilon\in \calE$ called mollifiers satisfying:
\begin{enumerate}
\item Every mollifier $\phi_{z,\epsilon}$ is a probability density with respect to $\lambda$, that is: $\forall y\in \calZ, \phi_{z,\epsilon}(y)\geq 0$ and $\int_\calZ\phi_{z,\epsilon}(y)d\lambda(y)=1$.
\item The mollifiers lie in ${\mathscr L}^2(\lambda)$ uniformly in $\calZ$. More precisely $\|\phi_{z,\epsilon}\|_{{\mathscr L}^2(\lambda)}<\infty$ and for every $\epsilon\in \calE$ there exist positive constants $c_\epsilon,C_\epsilon$ such that 
\[ \forall z\in \calZ, c_\epsilon<\|\phi_{z,\epsilon}\|_{{\mathscr L}^2(\lambda)}< C_\epsilon.\]
\item The mollifiers $\phi_{z,\epsilon}$ reproduce the evaluation at $z$ as $\epsilon\in \calE$ approaches zero:
\[\forall z\in \calZ,\; \forall p\in C^0(\calZ), \;\lim_{\epsilon\rightarrow 0} \int_{\calZ} \phi_{z,\epsilon}(t)p(y)d\lambda(y)=p(x).\]

\item The mollifiers  $\phi_{z,\epsilon}$ concentrate their energy around $z$ as $\epsilon\in \calE$ approaches zero:
\[\forall \delta>0,\; \lim_{\epsilon\rightarrow 0} \frac{\int_{\calZ \cap \{y:|y-x|>\delta\}} \phi_{z,\epsilon}(y)^2d\lambda(y)}{\|\phi_{z,\epsilon}\|^2_{{\mathscr L}^2(\lambda)}}=0.\]
\end{enumerate}
\end{definition}

Henceforth assume $\calZ$ is a variety endowed with ${\mathscr L}^2$-mollifiers. For a positive integer $d$, let $V_d$ be the vector space of functions on $\calZ$ that arise as restriction  to $\calZ$ of some polynomial with real coefficients of degree at most $d$ in $\RR^n$.

\begin{definition} Assume $\mu$ is a Borel measure with support $X:=\supp\mu \subseteq \calZ$ and $A\subseteq \calZ$ is any Borel set. For $\epsilon\in \calE$ and a positive integer $d$ the mollified Christoffel-Darboux (MCD) kernel of degree $d$ and resolution $\epsilon$ determined by $\mu$ and $A$ is the function $MCD_{d,\epsilon}^{\mu,A}: \calZ\times \calZ\rightarrow\RR$ given by
\[MCD_{d,\epsilon}^{\mu,A}(x,y)=\langle \varphi_x,\varphi_y\rangle _{{\mathscr L}^2(\mu)}\]
where, for each $z\in \calZ$, the function $\varphi_z\in V_d$ is the Riesz representative of the mollification centered at $z$ with resolution $\epsilon$ on $V_d\subseteq {\mathscr L}^2(\mu)$. More precisely, for each $z\in \calZ$, $\varphi_z$ is the element of $V_d$ satisfying
\[\forall p\in V_d,\; \int_A\phi_{z,\epsilon}(y)p(y)d\lambda(y) = \int_X \varphi_z(y)p(y)d\mu(y).\]
Note that although $\varphi_z$ depends not only on $z$ but also on $\mu,A,d$ and $\epsilon$, we omit those dependencies for notational simplicity.
\end{definition}

The following Lemma summarizes the basic properties of MCD kernels, analogous to well-known properties of the standard CD kernel. Since we will use them freely we include short and simple proofs for the reader's benefit. Furthermore we introduce the following notation: for any function $p\in C^0(\calZ)$ let 
\[\ell^A_{z,\epsilon}(p):=\int_A \phi_{z,\epsilon}(y)p(y)d\lambda(y).\]
Finally, recall that a set $A\subseteq \calZ$ is Zariski-dense in $\calZ$ if every polynomial vanishing identically on $A$ vanishes identically on $\calZ$. 
In subsequent sections, we specialize $A$ to be  $\calZ$ or $X$ in order to achieve support detection and density recovery, respectively.

\begin{lemma} \label{lem: CD_kernel_basics} If $X$ is compact and Zariski dense in $\calZ$ then the following statements hold for any $\epsilon\in \calE$ and any positive integer $d$,
\begin{enumerate}
\item The MCD kernel $MCD_{d,\epsilon}^{\mu,A}(x,y)$ is well-defined.
\item The function $MCD_{d,\epsilon}^{\mu,A}(x,y)$ can be computed via either of the following two formulas:
\begin{enumerate}
\item If $q_1,\dots, q_N$ are a $\mu$-orthonormal basis for $V_d$ then
\[MCD_{d,\epsilon}^{\mu,A}(x,y) = \sum_{j=1}^N\ell^A_{x,\epsilon}(q_j)\ell^A_{y,\epsilon}(q_j)\]
\item If $b_1,\dots, b_N$ are any basis for $V_d$ then 
\[MCD_{d,\epsilon}^{\mu,A}(x,y) = \ell^A_{x,\epsilon}(b)^\top M^{-1} \ell^A_{y,\epsilon}(b)\]
where $\ell^A_{z,\epsilon}(b)$ is the column vector with components  $\ell^A_{z,\epsilon}(b_j)$ and $M$ is the symmetric moment matrix with entries $M_{ij}:=\int_X b_i(y)b_j(y)d\mu(y)$.
\end{enumerate}
\item The diagonal of the mollified CD kernel $MCD_{d,\epsilon}^{\mu,A}(x,x)$ and its multiplicative inverse, known as the mollified Christoffel function, admit the following variational characterizations:
\begin{enumerate}
\item \[MCD_{d,\epsilon}^{\mu,A}(x,x)=\max\left\{ \ell_{x,\epsilon}^A(p)^2: p\in V_d\text{ , }\|p\|_{{\mathscr L}^2(\mu)}=1\right\}\]
\item \[\frac{1}{MCD_{d,\epsilon}^{\mu,A}(x,x)}=\min\left\{ \|p\|_{{\mathscr L}^2(\mu)}^2: p\in V_d\text{ , }\ell_{x,\epsilon}^A(p)=1\right\}\]
\end{enumerate}
\end{enumerate}
\end{lemma}
\begin{proof} $(1)$ Since $X$ is compact the function $\langle p,q\rangle_{{\mathscr L}^2(\mu)}:=\int_X p(y)q(y)d\mu(y)$ is a bilinear symmetric form on $V_d$. If $\langle p,p\rangle _{{\mathscr L}^2(\mu)}=0$ for some $p\in V_d$ then, by continuity,  $p$ it is identical to zero on $X$ and thus on $\calZ$ by Zariski density. It follows that $V_d$ is a Hilbert space with the inner product $\langle p,q\rangle_{{\mathscr L}^2(\mu)}$ and thus any linear functional, and in particular those coming from mollified CD-kernels, have unique Riesz representatives proving that the function $MCD_{d,\epsilon}^{\mu,A}(x,y)$ is well defined.
$(2a)$ If $q_1,\dots, q_N$ is a $\mu$-orthonormal basis and $z\in \calZ$ then $\varphi_z:=\sum_{j=1}^N \ell_{z,\epsilon}^{A}(q_j) q_j$ so by orthonormality of the $q_j$, $MCD^{\mu,A}_{d,\epsilon}(x,y)=\langle \varphi_x,\varphi_y\rangle_{{\mathscr L}^2(\mu)}$
is given by the formula in $(2a)$.$(2b)$ Since $q$ is an orthonormal basis the column vectors $b$ and $q$ containing the bases are related by the matrix product $b=Uq$ where $U_{ij}=\langle b_i,q_j\rangle_{{\mathscr L}^2(\mu)}$. Integrating entry by entry the resulting relation $bb^\top= Uqq^\top U^\top$ we conclude that the moment matrix satisfies $M=UU^\top$. It follows that
\[\ell^A_{x,\epsilon}(b)^\top M^{-1} \ell^A_{y,\epsilon}(b) = (U\ell^A_{x,\epsilon}(q))^\top (UU^\top)^{-1} U\ell^A_{y,\epsilon}(q)= \ell^A_{x,\epsilon}(q)^\top  \ell^A_{y,\epsilon}(q)=MCD^{\mu,A}_{d,\epsilon}(x,y)\]
where the last equality follows from $(2a)$. $(3a)$ By definition, we know that $MCD_{d,\epsilon}^{\mu,A}(x,x)=\|\varphi_x\|^2_{{\mathscr L}^2(\mu)}$ so the variational characterization of every norm in a Hilbert space  
\[\|\varphi_x\|_{{\mathscr L}^2(\mu)}=\max\left\{ \langle \varphi_x,p\rangle_{{\mathscr L}^2(\mu)}: p\in V_d\text{ , }\|p\|_{{\mathscr L}^2(\mu)}=1\right\}\]
immediately implies the validity of the formula $(3a)$.
For $(3b)$ note that if $1=\ell_{x,\epsilon}^A(p)=\langle \varphi_x,p\rangle$ then the Cauchy-Schwartz inequality implies $1\leq \|\varphi_x\|^2_{{\mathscr L}^2(\mu)}\|p\|^2_{{\mathscr L}^2(\mu)}$ so
\[\frac{1}{MCD_{d,\epsilon}^{\mu,A}(x,x)}\leq \min\left\{ \|p\|_{{\mathscr L}^2(\mu)}^2: p\in V_d\text{ , }\ell_{x,\epsilon}^A(p)=1\right\}.\]
If $p^*$ is a witness for the norm of $\varphi_x$ in the sense that $\|p^*\|_{{\mathscr L}^2(\mu)}=1$ and $\langle \varphi_x,p^*\rangle = \|\varphi_x\|$ then $p:=p^*/\ell_{x,\epsilon}^A(p^*)$ verifies that the inequality above is in fact an equality.
\end{proof}

\subsection{Examples of varieties with mollifiers}

The reader should keep in mind the following examples while reading the general theory. They are the simplest examples where the theory can be applied. We will revisit them later on to prove detailed quantitative estimates on their mollified CD kernels. Let $\mathscr B(x,r)$ denote the Euclidean ball of radius $r$ centered around $x$. 

\begin{example}\label{example: Rn} (Euclidean space with local support mollifiers) Let $\calZ=\RR^n$ and let $\lambda$ be the Lebesgue measure on $\calZ$. Let $\phi:\calZ\rightarrow \RR$ be a nonnegative function with $\int_{\calZ}\phi(y)d\lambda(y)=1$ and $\supp\phi\subseteq \mathscr B(0,1)$. For $z\in \calZ$ and $\epsilon>0$ define
\[\phi_{z,\epsilon}(y):=\frac{1}{\epsilon^n} \phi\left(\frac{x-y}{\epsilon}\right).\]
These functions have \emph{local support} in the sense that $\supp\phi_{z,\epsilon}\subseteq \mathscr B(z,\epsilon)$ and integrate to one by the chosen scaling. It follows easily that the $\{\phi_{z,\epsilon}\}$ are mollifiers for $\calZ$ with $\calE:=(0,+\infty)$. 

Two special choices of $\phi$ will play an important role in the article:
\begin{enumerate}
\item \emph{Lasserre's modified mollifiers}~\cite{Lasserre2023}
obtained by setting
\[\phi(y):=\left(\frac{\sqrt{n}}{2}\right)^n\mathbf{1}_{\left\{\|y\|_{\infty}<\frac{1}{\sqrt{n}}\right\}}(y).\]
\item {\it $C^{\infty}$ mollifiers}, obtained by setting
\[\phi(y):=\begin{cases}
\exp\left(\frac{-1}{1-\|y\|^2}\right)c^{-1}\text{, if $\|y\|<1$}\\
0\text{ , otherwise}
\end{cases}
\]
where $c:=\displaystyle\int_{\|y\|<1}\exp\left(\frac{-1}{1-\|y\|^2}\right)d\lambda(y)$ is the normalizing constant.
\end{enumerate}
\end{example}

\begin{example} \label{example:sphere} (Algebraic mollifiers on the sphere) Let $\calZ:=\calS$ be the standard unit sphere in $\RR^n$ and let $\lambda$ be its normalized surface area measure (i.e. the unique rotation-invariant probability measure). Following Ragozin~\cite{Ragozin1971} one can construct a sequence of nonnegative univariate functions $\psi_k: [-1,1]\rightarrow \RR$, $k\in \mathbb{N}$ such that the zonal functions $g_{x,k}(y):=\psi_k(\langle x,y\rangle)$ form an \emph{approximate identity on the sphere}, in the sense that
\[
\forall h\in C^0(\calZ),\;\lim_{k\rightarrow \infty} \frac{\int_\calZ g_{x,k}(y)h(y)d\lambda(y)}{\int_\calZ g_{x,k}(y)d\lambda(y)}=h(x).
\] 
Moreover, one can arrange that the squares $g_{x,k}(y)^2$ also form an approximate identity. 
It follows that the quantity
\[\epsilon_k:=\frac{\int_\calZ g_{x,k}(y)\|x-y\|^2d\lambda(y)}{\int_\calZ g_{x,k}(y)d\lambda(y)}\]
is independent of $x$ and converges to zero as $k\rightarrow \infty$.

Letting $\calE:=\{\epsilon_k: k\in \mathbb{N}\}$ and defining, for $\epsilon\in \calE$ and $x\in \calZ$ the function
\[\phi_{x,\epsilon}(y):=\frac{g_{x,\epsilon}(y)}{\int_{\calZ} g_{x,\epsilon}(y)}\text{ , if $\epsilon=\epsilon_k$}\]
we endow the sphere $\calZ$ with a collection of ${\mathscr L}^2$-mollifiers which satisfy the conditions of Definition~\ref{def: space}.
We call the mollifiers $g_{x,k}$ \emph{algebraic} if the functions $\psi_k(t)$ are univariate polynomials.
\end{example}

\begin{remark} The previous two examples illustrate the diversity of situations that a satisfactory theory for mollifiers on varieties needs to account for. It must admit both continuous and discrete families of mollifiers and it must allow for mollifiers with local support and for mollifiers supported in all of $\calZ$ (the latter is necessary if one wishes to consider algebraic mollifiers because no nonzero regular function on an irreducible variety can be supported inside a small euclidean ball).
\end{remark}

\section{Mollified kernels for support location and density estimation}
\label{sec: SL_and_DE}

Motivated by applications, we will focus on two main special cases of the MCD kernel defined previously: the \emph{Support locator MCD kernel}, defined as 
\[SMCD^{\mu}_{d,\epsilon}(x,y):=MCD^{\mu,\calZ}_{d,\epsilon}(x,y),\] and
the \emph{Density estimator MCD kernel}, defined as
\[DMCD^{\mu}_{d,\epsilon}(x,y):=MCD^{\mu,X}_{d,\epsilon}(x,y).\]
The main results of this Section are two Theorems proving that under rather general assumptions the Support locator MCD satisfies an on-support/off support dichotomy property which allows us to locate the support of the measure $\mu$ and that a suitably normalized limit of the Density estimator MCD converges to the Radon-Nikodym density $\frac{d\mu}{d\lambda}$ at interior points of $X$. These results hold under the following basic assumptions which we make throughout.\\

{\bf Assumptions.} Let $\calZ\subseteq \RR^n$ be a variety with ${\mathscr L}^2$-mollifiers $\phi_{z,\epsilon}$ with respect to a reference measure $\lambda$. Let $\mu$ be a finite Borel measure on $\calZ$ and let $X\subseteq \calZ$ be its support. We assume that $X$ is a compact subset with nonempty euclidean interior relative to $\calZ$, that $X$ is Zariski dense on $\calZ$ and that $\mu$ admits a continuous and strictly positive density $f$ with respect to the measure $\lambda$ on $X$.

\begin{remark} Since $X$ has nonempty Euclidean interior relative to $\calZ$ it is automatically Zariski dense whenever the interior of $X$ intersects every irreducible component of $\calZ$. In particular, Zariski density holds automatically when the variety $\calZ$ is irreducible. 
\end{remark}

\begin{theorem} {\it (Improved support dichotomy)}.  If the mollifiers on $\calZ$ have local support, in the sense that $\forall z\in \calZ,\;\forall \epsilon\in \calE,\;\supp\phi_{z,\epsilon}\subseteq \mathscr B(z,\epsilon)$, then the following statements hold:
\begin{enumerate}
\item If $z\not\in X$ then $SMCD^{\mu}_{d,\epsilon}(z,z)$ grows exponentially in $d$ for every sufficiently small $\epsilon\in \calE$. More precisely, the following inequality holds 
\[SMCD^{\mu}_{d,\epsilon}(z,z) \geq \left(1+\frac{3\delta(z)^2}{4\left(\rho(z)^2-\delta(z)^{2}\right)}\right)^{2\lceil\frac{d}{2}\rceil} \frac{\left(\int_{\mathscr B(z,\delta(z)/2)}\phi_{z,\epsilon}(y)d\lambda(y)\right)^2}{\mu(X)}.\]
where $\delta(z):=\min\{\|z-y\|: y\in X\}$ is the distance, $ \rho(z):=\max\{\|z-y\|: y\in X\}$ is the eccentricity and $\epsilon<\delta(z)/2$.
\item If $z$ is interior to $X$ then $SMCD^{\mu}_{d,\epsilon}(z,z)$ is bounded in $d$ for every sufficiently small $\epsilon\in \calE$. More precisely, 
if $\mathscr B(z,\epsilon)\cap X=\mathscr B(z,\epsilon)\cap \calZ$ for some $\epsilon\in \calE$ then the function $SMCD^{\mu}_{d,\epsilon}(z,z)$ is a bounded function of $d$ with a uniform upper bound over all such $z$. 
\end{enumerate}

\end{theorem}
\begin{proof} $(1)$ To prove the lower bound in $(1)$ we will use the variational characterization from Lemma~\ref{lem: CD_kernel_basics} part $(3a)$. For $z\not\in X$ let \[p_z(y):=\left(1-\frac{\|z-y\|^2}{\rho(z)^2}\right)^{d}.\]
Since $z\not\in X$, for every $y\in X$ we have
\[\delta(z)\leq \|z-y\|\leq \rho(z)\]
and therefore
\[\left(1-\left(\frac{\delta(z)}{\rho(z)}\right)^2\right)^d\geq p_z(y)\geq 0\]
so in particular $\|p_z\|^2_{{\mathscr L}^2(\mu)}\leq \left(1-\left(\frac{\delta(z)}{\rho(z)}\right)^2\right)^{2d} \mu(X)$. Dividing by the right-hand side we conclude that the normalized polynomial
\[\widehat{p_z}(y):=\frac{p_z(y)}{\left(1-\left(\frac{\delta(z)}{\rho(z)}\right)^2\right)^{d} \sqrt{\mu(X)}}\]
lies in the unit ball in ${\mathscr L}^2(\mu)$ and therefore Lemma~\ref{lem: CD_kernel_basics} part $(3a)$ implies that $SMCD^{\mu}_{d,\epsilon}(z,z)\geq \ell_z^{\calZ}(\widehat{p_z}(y))^2$, where
\[\ell_z^{\calZ}(\widehat{p_z}(y))=\int_{\calZ} \phi_{z,\epsilon}(y)\widehat{p_z}(y)d\lambda(y).\]
If $\epsilon<\delta(z)/2$ then the function $\phi_{z,\epsilon}(y)\widehat{p_z}(y)$ is supported in $\mathscr B(z,\delta(z)/2)$. At all points $y$ of this ball $\|z-y\|\leq \delta(z)/2$ and therefore $p_z(y)\geq \left(1-\left(\frac{\delta(z)}{2\rho(z)}\right)^2\right)^{d}$. It follows that
\[\ell_z^{\calZ}(\widehat{p_z}(y))^2\geq \left(\frac{1-\left(\frac{\delta(z)}{2\rho(z)}\right)^2}{1-\left(\frac{\delta(z)}{\rho(z)}\right)^2}\right)^{2d} \frac{\left(\int_{\mathscr B(z,\delta(z)/2)}\phi_{z,\epsilon}(y)d\lambda(y)\right)^2}{\mu(X)}\] 
which yields the inequality in part $(1)$ for even $d$ by elementary algebraic manipulation. Since $V_d\subseteq V_{d+1}$ function $SMCD^{\mu}_{d,\epsilon}(z,z)$ is monotonic in $d$ proving the claimed inequality for odd $d$.

$(2)$ If $\mathscr B(z,\epsilon)\cap X=\mathscr B(z,\epsilon)\cap \calZ$ for some $\epsilon\in \calE$ then by the local support assumption of mollifiers we have
\[\int_\calZ \phi_{z,\epsilon}(y)p(y)d\lambda(y) = \int_X\phi_{z,\epsilon}(y)p(y)d\lambda(y).\]
The strict positivity of the density $f$ of $\mu$ with respect to $\lambda$ in $X$ implies that for every $p\in C^0(X)$
\[\int_X \phi_{z,\epsilon}(y)p(y)d\lambda(y)= \int_X \frac{\phi_{z,\epsilon}(y)}{f(y)}p(y)d\mu(y)\]
and in particular the Riesz representative in $V_d\subseteq {\mathscr L}^2(\mu)$ of the linear function defined by the left-hand side is given by the orthogonal projection $\Pi^{\mu}_{d}\left(\frac{\phi_{z,\epsilon}(y)}{f(y)}\right)$. 
It follows that
\[SMCD^{\mu}_{d,\epsilon}(z,z)=\left\|\Pi^{\mu}_{d}\left(\frac{\phi_{z,\epsilon}(y)}{f(y)}\right)\right\|_{{\mathscr L}^2(\mu)}^2\leq \left\|\frac{\phi_{z,\epsilon}(y)}{f(y)}\right\|_{{\mathscr L}^2(\mu)}^2 = \int_X \frac{\phi_{z,\epsilon}(y)^2}{f(y)}d\lambda(y)\leq M C_{\epsilon}\]
where $M=1/\min\{|f(x)|:x\in X\}>0$ and $C_{\epsilon}$ is the constant appearing in part $(2)$ of the definition of space with mollifiers and is therefore uniformly bounded as claimed.

\end{proof}

\begin{theorem}\label{Thm: densityrecovery} {\it (Density recovery)}. The following identity holds: 
\[\lim_{\epsilon\rightarrow 0}\left(\lim_{d\rightarrow \infty} \frac{DMCD^{\mu}_{d,\epsilon}(x,x)}{\|\phi_{x,\epsilon}\|^2_{{\mathscr L}^2(\lambda)}}\right)=
\begin{cases}
\frac{1}{f(x)}\text{ , if $x\in X^{\circ}$}\\
0\text{ , if $x\not\in X$.}\\
\end{cases}
\]

\end{theorem}
\begin{proof} Since $X$ is Zariski dense in $\calZ$, the DMCD kernel is well defined by Lemma~\ref{lem: CD_kernel_basics} part $(1)$. If $z\in \calZ$ then the strict positivity of the density $f$ of $\mu$ with respect to $\lambda$ in $X$ implies that for every $p\in C^0(X)$
\[\int_X \phi_{z,\epsilon}(y)p(y)d\lambda(y)= \int_X \frac{\phi_{z,\epsilon}(y)}{f(y)}p(y)d\mu(y)\]
and in particular the Riesz representative in $V_d\subseteq {\mathscr L}^2(\mu)$ of the linear function defined by the left-hand side is given by the orthogonal projection $\Pi^{\mu}_{d}\left(\frac{\phi_{z,\epsilon}(y)}{f(y)}\right)$.
As a result the following equality holds for any $z\in \calZ$,
\[DMCD_{d,\epsilon}^{\mu}(z,z)=\left\|\Pi^{\mu}_{d}\left(\frac{\phi_{z,\epsilon}(y)}{f(y)}\right)\right\|_{{\mathscr L}^2(\mu)}^2.\]
Since $\frac{\phi_{z,\epsilon}(y)}{f(y)}\in {\mathscr L}^2(\mu)$ it follows that
\[\lim_{d\rightarrow \infty}DMCD_{d,\epsilon}^{\mu}(z,z) = \left\|\frac{\phi_{z,\epsilon}(y)}{f(y)}\right\|^2_{{\mathscr L}^2(\mu)}=\int_X\left(\frac{\phi_{z,\epsilon}(y)}{f(y)}\right)^2d\mu(y)=\int_X \frac{\phi_{z,\epsilon}(y)^2}{f(y)}d\lambda(y)\]
and we conclude that  
\[\lim_{\epsilon\rightarrow 0}\frac{\lim_{d\rightarrow \infty} DMCD_{d,\epsilon}^{\mu}(z,z)
}{\|\phi_{z,\epsilon}\|^2_{{\mathscr L}^2(\lambda)}}=\lim_{\epsilon\rightarrow 0}\frac{\int_X\frac{1}{f(y)}\phi_{z,\epsilon}(y)^2 d\lambda(y)}{\int_\calZ \phi_{z,\epsilon}(y)^2d\lambda(y)}.\]
If $x$ is any point interior to $X$, then there exists a $\delta_0>0$ such that $\mathscr B(x,\delta_0)\cap \calZ=\mathscr B(x,\delta_0)\cap X$. For any $\eta>0$ the continuity of $1/f(y)$ on $X$ implies that there exists $\delta_1<\delta_0$ such that $|1/f(y)-1/f(x)|<\eta$ whenever $|y-x|<\delta_1$ and $y\in \calZ$. Letting $h_x(y):=1/f(y)-1/f(x)$ for $y\in X$ we conclude that
\[ \left|\frac{\int_X\frac{1}{f(y)}\phi_{x,\epsilon}(y)^2 d\lambda(y)}{\int_\calZ \phi_{x,\epsilon}(y)^2d\lambda(y)}-\frac{1}{f(x)}\right|=\left|\frac{\int_X h_x(y)\phi_{x,\epsilon}(y)^2 d\lambda(y)}{\int_\calZ \phi_{x,\epsilon}(y)^2d\lambda(y)}\right|\leq \]
\[\left|\frac{\int_{X\cap \left\{y: |y-z|<\delta_1\right\}} h_x(y)\phi_{x,\epsilon}(y)^2 d\lambda(y)+\int_{X\cap \left\{y: |y-z|>\delta_1\right\}} h_x(y)\phi_{x,\epsilon}(y)^2 d\lambda(y)}{\int_\calZ \phi_{x,\epsilon}(y)^2d\lambda(y)}\right|\leq \]
\[\leq \frac{\int_{\calZ\cap \left\{y: |y-z|<\delta_1\right\}}|h_x(y)|\phi_{x,\epsilon}(y)^2 d\lambda(y)}{\int_\calZ \phi_{x,\epsilon}(y)^2d\lambda(y)}+\frac{\int_{X\cap \left\{y: |y-z|>\delta_1\right\}}\left|h_x(y)\right|\phi_{x,\epsilon}(y)^2 d\lambda(y)}{\int_\calZ \phi_{x,\epsilon}(y)^2d\lambda(y)}\]
\[\leq \eta + 2M \frac{\int_{\calZ\cap \left\{y: |y-z|>\delta_1\right\}}\phi_{x,\epsilon}(y)^2 d\lambda(y)}{\int_\calZ \phi_{x,\epsilon}(y)^2d\lambda(y)}\]
where $M$ denotes the maximum value of the continuous function $1/f(y)$ for $y$ in the compact set $X$. Since $|h_x(y)|\leq 2M$ and the mollifiers $\phi_{x,\epsilon}$ concentrate their energy around $x$ as $\epsilon\rightarrow 0$ (see property $(4)$ of Definition \ref{def: space} of the space with ${\mathscr L}^2$ mollifiers $\calZ$) the limit as $\epsilon\rightarrow 0$ of the last upper bound above converges to $\eta$. Since $\eta>0$ was arbitrary we conclude that 
\[\lim_{\epsilon\rightarrow 0}\left|\frac{\int_X\frac{1}{f(y)}\phi_{x,\epsilon}(y)^2 d\lambda(y)}{\int_\calZ \phi_{x,\epsilon}(y)^2d\lambda(y)}-\frac{1}{f(x)}\right|=0\]
proving the claimed equality for points interior to $X$.

If $z\not\in X$ then for any $d>0$ the following inequality holds,
\[DMCD_{d,\epsilon}^{\mu}(z,z)=\left\|\Pi^{\mu}_{d}\left(\frac{\phi_{z,\epsilon}(y)}{f(y)}\right)\right\|_{{\mathscr L}^2(\mu)}^2\leq \left\|\frac{\phi_{z,\epsilon}(y)}{f(y)}\right\|_{{\mathscr L}^2(\mu)}^2= \int_X\frac{1}{f(y)}\phi_{z,\epsilon}(y)^2 d\lambda(y)\]
we conclude that if $M$ denotes the maximum value of the continuous function $1/f(y)$ for $y$ in the compact set $X$ then 
\[\frac{DMCD_{d,\epsilon}^{\mu}(z,z)}{\|\phi_{z,\epsilon}(y)\|^2_{{\mathscr L}^2(\lambda)}}\leq \frac{\int_X\frac{1}{f(y)}\phi_{z,\epsilon}(y)^2 d\lambda(y)}{\|\phi_{z,\epsilon}(y)\|^2_{{\mathscr L}^2(\lambda)}}\leq M \frac{\int_X\phi_{z,\epsilon}(y)^2 d\lambda(y)}{\|\phi_{z,\epsilon}(y)\|^2_{{\mathscr L}^2(\lambda)}}\leq M \frac{\int_{\calZ\cap \left\{y: |y-z|\geq {\rm dist}(z,X)\right\}}\phi_{z,\epsilon}(y)^2 d\lambda(y)}{\|\phi_{z,\epsilon}(y)\|^2_{{\mathscr L}^2(\lambda)}}\]
where the last inequality holds because $z\not\in X$ and thus $X\subseteq \calZ\cap \left\{y: |y-z|\geq {\rm dist}(z,X)\right\}$.  The rightmost quantity goes to zero as $\epsilon\rightarrow 0$ because the mollifiers $\phi_{z,\epsilon}$ concentrate their energy around $z$ as $\epsilon\rightarrow 0$ (by property $(4)$ of the space with ${\mathscr L}^2$ mollifiers $\calZ$).\end{proof}

\section{Quantitative estimates for density recovery via MCD kernels}
\label{sec: Quant}

As shown in Theorem~\ref{Thm: densityrecovery} the density estimator MCD kernel converges pointwise to $1/f(x)$ at points $x$ interior to $X$. In this section we show that under Sobolev regularity assumptions on the density $f(x)$, the convergence is in fact uniform if $\epsilon$ and $d$ are coupled appropriately. More precisely, in the next two sections we derive quantitative estimates on the rate of convergence in two cases of much interest to applications:
\begin{enumerate}
    \item Recovery of densities on regular compact domains in Euclidean space using local support mollifiers (see Theorem~\ref{thm:rate-general}) and
    \item Recovery of densities via algebraic mollifiers on spheres (see Theorem~\ref{thm:rate-sphere}).    
\end{enumerate}
While the details of the estimations depend on the chosen space, both proofs share a fundamental structural similarity which we now emphasize.

For each $x\in\interior\X$ and $\varepsilon>0$, we have seen that if
\[
h_{x,\varepsilon}(y) := \frac{\phi_{x,\varepsilon}(y)}{f(y)}\in {\mathscr L}^2(\mu)
\]
then the Density estimator mollified CD polynomial is given by
\[
{\rm DMCD}^{\mu}_{d,\epsilon}(x,x) = \norm{\Pi_d h_{x,\varepsilon}}_{{\mathscr L}^2(\mu)}^2.
\]

The key observation, originally made in  \cite{Lasserre2023}, is that in the limit of infinite-dimensional approximation, the quantity
\[
\frac{\norm{h_{x,\varepsilon}}_{{\mathscr L}^2(\mu)}^2}{\|\phi_{x,\varepsilon}\|_{{\mathscr L}^2(\lambda)}^2}
= \frac{\displaystyle\int_{\X} \frac{\phi_{x,\varepsilon}(y)^2}{f(y)}\,d\mu(y)}{\displaystyle\int_{\X} \phi_{x,\varepsilon}(y)^2\,d\lambda(y)}
\]
converges to $1/f(x)$ as $\varepsilon\to0$. Thus our estimator of $1/f(x)$ is
\begin{equation}
	\label{eq:estimator}
	\widehat{g}_{d,\varepsilon}(x)
	:= \frac{p_{d,\varepsilon}^\mu(x)}{\|\phi_{x,\varepsilon}\|_{{\mathscr L}^2(\lambda)}^2}
	= \frac{\norm{\Pi_d h_{x,\varepsilon}}_{{\mathscr L}^2(\mu)}^2}{\|\phi_{x,\varepsilon}\|_{{\mathscr L}^2(\lambda)}^2},
	\qquad g(x):=\frac{1}{f(x)}.
\end{equation}

We decompose the error as
\begin{equation}
	\label{eq:error-decomposition}
	\left| \widehat{g}_{d,\varepsilon}(x) - g(x) \right|
	\le
	\underbrace{
		\frac{\left| \norm{\Pi_d h_{x,\varepsilon}}_{{\mathscr L}^2(\mu)}^2 - \norm{h_{x,\varepsilon}}_{{\mathscr L}^2(\mu)}^2 \right|}
		{\|\phi_{x,\varepsilon}\|_{{\mathscr L}^2(\lambda)}^2}
	}_{\text{projection error}}
	+
	\underbrace{
		\left|
		\frac{\norm{h_{x,\varepsilon}}_{{\mathscr L}^2(\mu)}^2}{\|\phi_{x,\varepsilon}\|_{{\mathscr L}^2(\lambda)}^2}
		- g(x)
		\right|}_{\text{approximation error}}.
\end{equation}

The projection error is due to finite-dimensional approximation ($V_d$ versus the full ${\mathscr L}^2(\mu)$), while the approximation error is due to mollification ($\varepsilon>0$ versus $\varepsilon\to0$). Our quantitative estimates bound each of these sources of error independently.

\subsection{Density recovery in Euclidean space using mollifiers with local support}
\label{sec: QuantRn}

We now turn to quantitative estimates on density recovery in Euclidean space $\RR^n$ using mollifiers with local support as in Example~\ref{example: Rn}. More precisely, our Theorems will hold under the following set of mild assumptions:

\begin{assumption}\label{assumpt:Rn}
The following statements hold:
\begin{enumerate}
    \item The chosen mollifiers are given by
    \[\phi_{z,\epsilon}(y)=\frac{1}{\epsilon^n}\phi\left(\frac{z-y}{\epsilon}\right)\]
    where $\phi:\R^n\rightarrow \R$ is a nonnegative, $C^{\infty}$ function with $\int_{\R^n}\phi(y)d\lambda(y)=1$ and $\;\supp  \phi\subseteq B(0,1)$. These mollifiers have local support in the sense that $\forall z\in \RR^n,\;\forall \epsilon\in \calE,\; \supp\phi_{z,\epsilon}\subseteq \mathscr B(z,\epsilon).$
    \item The chosen mollifiers $\phi_{z,\epsilon}$ are symmetric around $z$ in the sense that if $\mathscr B(z,\epsilon)\subseteq X$ then
    \[0=\int_X \phi_{z,\epsilon}^2(y) \eta(y)d\lambda(y)\]
    holds for every linear form $\eta(y)=\sum c_i (y_i-z_i)$ vanishing at $z$. This assumption holds whenever $\phi(-y)=\phi(y)$ for every $y\in \R^n$.
\item The compact set $X$ which supports the measure $\mu$ is a regular Lipschitz domain. This means that it is equal to the closure of its interior and that its boundary is given locally by the graph of a Lipschitz function. 
\item The density is bounded and strictly positive. More precisely we assume $0<m\le f\le M < \infty$ on $\X$ for given constants $m$ and $M$.
\end{enumerate}
\end{assumption}

Parts $(1)$ and $(2)$ of the previous assumption set hold for the mollifiers in Example~\ref{example: Rn} and do so for any set of mollifiers for which the basic function $\phi$ is radially symmetric. Part $(3)$ holds quite generally, for instance it holds for any compact set with $\mathscr C^1$ boundary.

\subsubsection{Approximation error}

We first control the approximation error, which depends only on the local regularity of the density.

\begin{lemma}[Approximation error]
	\label{lem:approx-error}
	Assume $f \in \mathscr C^2(\X)$. If Assumption~\ref{assumpt:Rn} holds, then for each compact set
	$K\subset\interior\X$ there exist constants $C>0$ and
	$\varepsilon_0>0$ such that for all $x\in K$ and all
	$0<\varepsilon<\varepsilon_0$,
	\[
	\left|
	\frac{\|h_{x,\varepsilon}\|_{{\mathscr L}^2(\mu)}^2}{\|\phi_{x,\varepsilon}\|_{{\mathscr L}^2(\lambda)}^2}
	- \frac{1}{f(x)} \right|
	\le C\,\varepsilon^2.
	\]
	In particular, the convergence is uniform on $K$.
\end{lemma}

\begin{proof}
	By definition,
	\[
	\frac{\|h_{x,\varepsilon}\|_{{\mathscr L}^2(\mu)}^2}{\|\phi_{x,\varepsilon}\|_{{\mathscr L}^2(\lambda)}^2}
	= \frac{\displaystyle\int_{\X} \phi_{x,\varepsilon}(y)^2 g(y)\,d\lambda(y)}
	{\displaystyle\int_{\X} \phi_{x,\varepsilon}(y)^2\,d\lambda(y)}.
	\]
	
	Fix a compact set $K\subset\interior\X$. Because $g\in \mathscr C^2(\X)$, for
	each $x\in K$ there is a quadratic Taylor expansion
	\[
	g(y)
	= g(x) + (y-x)^T \nabla g(x) 
	+ \frac12(y-x)^T \nabla^2 g(x)(y-x) + r_x(y),
	\]
	where $\nabla g(x)$ is the gradient and $\nabla^2 g(x)$ is the Hessian of $g$ at $x$. For $y$ close to $x$,
	\[
	|r_x(y)| \le C_1\|y-x\|^2,
	\]
	with a constant $C_1$ independent of $x\in K$ (by compactness).
	
	Since $K\subset\interior\X$ and $\phi_{x,\varepsilon}$ has compact support,
	we can choose $\varepsilon_0>0$ such that, for all $x\in K$ and
	$0<\varepsilon<\varepsilon_0$, the support of $y\mapsto\phi_{x,\varepsilon}(y)$
	is contained in a ball $\mathscr B(x,\varepsilon)\subset\X$ where the Taylor
	estimate above holds. Substituting the Taylor expansion into the ratio gives
    \begin{align*}
    \frac{\|h_{x,\varepsilon}\|_{{\mathscr L}^2(\mu)}^2}{\|\phi_{x,\varepsilon}\|_{{\mathscr L}^2(\lambda)}^2}
    &= g(x)
    + \frac{\displaystyle\int_\X \phi_{x,\varepsilon}(y)^2
            \bigl[(y-x)^T \nabla g(x)\bigr]\,d\lambda(y)}
           {\displaystyle\int_\X \phi_{x,\varepsilon}(y)^2\,d\lambda(y)} \\
    &\quad + \frac12 \frac{\displaystyle\int_\X \phi_{x,\varepsilon}(y)^2
            \bigl[(y-x)^T \nabla^2 g(x)(y-x) + e_x(y)\bigr]\,d\lambda(y)}
           {\displaystyle\int_\X \phi_{x,\varepsilon}(y)^2\,d\lambda(y)}.
    \end{align*}
	
	The first correction term vanishes by symmetry (part $(2)$ of Assumption~\ref{assumpt:Rn}), yielding
	\[
	\int_\X \phi_{x,\varepsilon}(y)^2 \bigl[(y-x)^T \nabla g(x)\bigr]\,d\lambda(y) = 0.
	\]
	
	For the quadratic term, let
	\[
	\Lambda := \sup_{x\in K, z} \left|\frac{z^\top  \nabla^2 g(x)z}{z^\top z}\right| \in (0,\infty)
	\]
	so that
	\[
	\bigl|(y-x)^T \nabla^2 g(x)(y-x)\bigr|
	\le \Lambda \|y-x\|^2.
	\]
	Using also $|r_x(y)|\le C_1\|y-x\|^2$, we get the uniform bound
	\[
	\bigl|(y-x)^T \nabla^2 g(x)(y-x) + r_x(y)\bigr|
	\le (\Lambda+C_1)\,\|y-x\|^2.
	\]
	By part $(1)$ of Assumption~\ref{assumpt:Rn}, $\phi_{x,\varepsilon}$ is supported in $\mathscr B(x,\varepsilon)$ so we have
	$\|y-x\|\le\varepsilon$ on the support, and therefore
	\[
	\left|
	\frac{\displaystyle\int_\X \phi_{x,\varepsilon}(y)^2
		\bigl[(y-x)^T \nabla^2 g(x)(y-x) + r_x(y)\bigr]\,d\lambda(y)}
	{\displaystyle\int_\X \phi_{x,\varepsilon}(y)^2\,d\lambda(y)}
	\right|
	\le (\Lambda+C_1)\,\varepsilon^2.
	\]	
	Combining the above estimates yields
	\[
	\left|
	\frac{\|h_{x,\varepsilon}\|_{{\mathscr L}^2(\mu)}^2}{\|\phi_{x,\varepsilon}\|_{{\mathscr L}^2(\lambda)}^2}
	- g(x)
	\right|
	\le C\,\varepsilon^2
	\]
	for all $x\in K$ and $0<\varepsilon<\varepsilon_0$, with
	$C:=\frac12(\Lambda+C_1)$ independent of $x$ and $\varepsilon$. This proves the claim
	and shows that the convergence is uniform on $K$.
\end{proof}

\begin{remark}
	The previous proof shows that the approximation error is controlled by the variance
	\[
	\frac{\displaystyle\int_\X \phi_{x,\varepsilon}(y)^2 \|y-x\|^2\,d\lambda(y)}
	{\displaystyle\int_\X \phi_{x,\varepsilon}(y)^2\,d\lambda(y)}.
	\]
	For standard scaled radial mollifiers this quantity behaves like
	$c\,\varepsilon^2$, so the rate $\mathcal O(\varepsilon^2)$ cannot be improved uniformly over all $\mathscr C^2$ densities $f$.
	Faster rates are in principle possible only if one allows signed, higher-order
	kernels with additional vanishing-moment conditions, which are no longer
	positive mollifiers of the usual type.
\end{remark}

\subsubsection{Projection error and Sobolev approximation}

The projection error involves the approximation of $h_{x,\varepsilon}$ by its
orthogonal projection onto $V_d$. We now quantify this error using polynomial
approximation in Sobolev spaces, following \cite{LiXu2014}.

We first recall the relevant Sobolev estimate on the ball. Let
${\mathscr B}\subset\R^n$ denote a closed ball, and for an
integer $k\ge0$ let $\mathscr H^k(\mathscr B)$ be the Sobolev space of functions
$h:{\mathscr B}\to\R$ whose weak derivatives $\partial^\alpha h$ belong
to ${\mathscr L}^2({\mathscr B})$ for all multi-indices $\alpha$ with $|\alpha|\le k$.
The Sobolev norm is
\[
\|h\|_{\mathscr H^k({\mathscr B})}
:= \sqrt{\sum_{|\alpha|\le k} \|\partial^\alpha h\|_{{\mathscr L}^2({\mathscr B})}^2}.
\]

\begin{theorem}\label{thm:LiXu}~\cite[Theorem~1.1]{LiXu2014}
	For each integer $k\ge1$ there exists a constant $c_k>0$ such that for every
	$h\in \mathscr H^k({\mathscr B})$ and every integer $d\ge1$ there exists a polynomial
	$p_d$ of degree at most $d$ satisfying
	\[
	\|h-p_d\|_{{\mathscr L}^2({\mathscr B})}
	\le c_k\,d^{-k}\,\|h\|_{\mathscr H^k({\mathscr B})}.
	\]
	In particular, the best ${\mathscr L}^2$ approximation error by degree-$d$ polynomials is
	bounded by the right-hand side.
\end{theorem}

\begin{remark}
	The statement is a specialization of~\cite[Theorem~1.1]{LiXu2014} with
	$p=2$, $k=0$ in their notation and our $k$ playing the role of their smoothness
	index $s$. Their parameter $n$ corresponds to the polynomial degree $d$.
\end{remark}

We now translate this into a bound on the projection error which will depend on the Sobolev regularity of the density,

\begin{assumption}
	\label{ass:sobolev}
	Assume that $f$ and $1/f$ belong to $\mathscr H^s(\X)$ for some $s>n/2$.
\end{assumption}

\begin{lemma}[Projection error]\label{lem:proj-error}
	Under Assumption~\ref{assumpt:Rn} and Assumption~\ref{ass:sobolev}, fix an integer $k$ with $1\le k\le s$ and $k>n/2$. For every compact set $K\subseteq \interior X$ there exists positive constants $\varepsilon_0, C$ such that the following equality holds for every $x\in K$, $\epsilon\leq \epsilon_0$ and positive integer $d$,
    \begin{equation}\label{eq:proj-L2mu}
		\big\| h_{x,\varepsilon} - \Pi_d h_{x,\varepsilon}\big\|_{{\mathscr L}^2(\mu)}^2
		\;\le\;
		C(\phi,f,k,n)\,\varepsilon^{-n}(\varepsilon d)^{-2k}.
	\end{equation}
	Consequently, the difference of squared norms satisfies
	\begin{equation}\label{eq:proj-sqdiff}
		0\le
		\big\|h_{x,\varepsilon}\big\|_{{\mathscr L}^2(\mu)}^2
		- \big\|\Pi_d h_{x,\varepsilon}\big\|_{{\mathscr L}^2(\mu)}^2
		= \big\|h_{x,\varepsilon} - \Pi_d h_{x,\varepsilon}\big\|_{{\mathscr L}^2(\mu)}^2
		\;\le\;
		C(\phi,f,k,n)\,\varepsilon^{-n}(\varepsilon d)^{-2k}.
	\end{equation}

\end{lemma}

\begin{proof}
	\emph{Step 1: Localization and extension to a ball.}
	Since $\X$ is bounded, there exists a closed ball ${\mathscr B}\subset\R^n$
	such that $\X\subset{\mathscr B}$. since $X$ is a Lipschitz Domain, a Sobolev 
	extension Theorem~\cite[Theorem 12.15]{Leoni2009} guarantees the existence of an extension $\tilde{g
}\in \mathscr H^s({\mathscr B})$ which agrees with $g$ on $X$.

    Because $K$ is compact there exists a positive real number $\varepsilon_0$ such that 
    $\mathscr B(x,\varepsilon_0)
	\subset\interior X\subset{\mathscr B}$ for every $x\in K$. If
	$0<\varepsilon<\varepsilon_0$ then the support of
	$y\mapsto\phi_{x,\varepsilon}(y)$ is contained in the ball $\mathscr B(x,\varepsilon)\subseteq \interior X$ and therefore $\phi_{x,\epsilon}$ or more precisely its extension by zero lies in $ H^s({\mathscr B})$.
	
	\medskip\noindent
	\emph{Step 2: Sobolev regularity and scaling of the Riesz representer.} Since $k>\frac{n}{2}$ the spaces $H^k({\mathscr B})$ are Banach algebras so they are closed under product
	\[
	h_{x,\varepsilon}(\cdot) = \phi_{x,\varepsilon}(\cdot)\,\tilde{g}(\cdot)
	\in \mathscr H^k({\mathscr B}),
	\]
 and there exists a constant $C'(k,n)$ such that $\|h_{x,\varepsilon}\|_{\mathscr H^k({\mathscr B})}\leq C' \|\phi_{x,\varepsilon}\|_{\mathscr H^k({\mathscr B})}\|g\|_{\mathscr H^k({\mathscr B})}$.

	We now estimate the Sobolev norm of $\phi_{x,\varepsilon}$.
	For any multi-index $\alpha$ with $|\alpha|\le k$, the chain rule gives
	\[
	\partial^\alpha \phi_{x,\varepsilon}(y)
	= \varepsilon^{-n-|\alpha|}\,
	(\partial^\alpha \phi)\!\left(\frac{x-y}{\varepsilon}\right).
	\]

	Hence
	\[
	\|\partial^\alpha \phi_{x,\varepsilon}\|_{{\mathscr L}^2({\mathscr B})}^2
	\le \varepsilon^{-2n-2|\alpha|}
	\int_{\R^n}\left|\partial^\alpha\phi\!\left(\frac{x-y}{\varepsilon}\right)\right|^2\,dy
	= \varepsilon^{-2n-2|\alpha|}\,\varepsilon^n
	\int_{\R^n}|\partial^\alpha\phi(z)|^2\,dz
	= C_\alpha\,\varepsilon^{-n-2|\alpha|},
	\]
	with $C_\alpha:=\|\partial^\alpha\varphi\|_{{\mathscr L}^2(\R^n)}^2$ independent of
	$x$ and $\varepsilon$. Summing over $|\alpha|\le k$ yields
	\[
	\|\phi_{x,\varepsilon}\|_{\mathscr H^k({\mathscr B})}^2
	= \sum_{|\alpha|\le k}\|\partial^\alpha \phi_{x,\varepsilon}\|_{{\mathscr L}^2({\mathscr B})}^2
	\le C_1(\phi,k,n)\,\varepsilon^{-n-2k},
	\]
	for some constant $C_1$ and therefore
	\[
	\|\phi_{x,\varepsilon}\|_{\mathscr H^k({\mathscr B})}
	\le C_1(\phi,k,n)^{1/2}\,\varepsilon^{-n/2-k}.
	\]	
	Combining this with the product estimate above yields,
	\begin{equation}\label{eq:h-sobolev-norm}
		\|h_{x,\varepsilon}\|_{\mathscr H^k({\mathscr B})}
		\;\le\;
		C_2(\phi,f,k,n)\,\varepsilon^{-n/2-k},
	\end{equation}
	with $C_2$ independent of $x\in K$ and $\varepsilon\leq \varepsilon_0$.
	
	\medskip\noindent
	\emph{Step 3: Polynomial approximation on the ball.}
	By Theorem~\ref{thm:LiXu}, applied on ${\mathscr B}$, there exists a polynomial
	$p_d$ of degree at most $d$ such that
	\[
	\|h_{x,\varepsilon} - p_d\|_{{\mathscr L}^2({\mathscr B})}
	\le c_k\,d^{-k}\,\|h_{x,\varepsilon}\|_{\mathscr H^k({\mathscr B})}.
	\]
	Using \eqref{eq:h-sobolev-norm}, we obtain
	\begin{equation}\label{eq:hd-pd-L2-B}
		\|h_{x,\varepsilon} - p_d\|_{{\mathscr L}^2({\mathscr B})}
		\le C_3(\phi,f,k,n)\,d^{-k}\,\varepsilon^{-n/2-k},
	\end{equation}
	for some constant $C_3$ independent of $x\in K$, $\varepsilon\leq \varepsilon_0$ and $d$.
	
	\medskip\noindent
	\emph{Step 4: From ${\mathscr L}^2({\mathscr B})$ to ${\mathscr L}^2(\mu)$ and best approximation.}
	Since $0<m\le f\le M < \infty$ on $\X$, the norms $\|\cdot\|_{{\mathscr L}^2(\mu)}$ and
	$\|\cdot\|_{{\mathscr L}^2(\lambda)}$ are equivalent on $\X$:
	\[
	\|u\|_{{\mathscr L}^2(\mu)}^2 = \int_\X u^2 f\,d\lambda
	\le M\int_\X u^2\,d\lambda
	\le M\int_{{\mathscr B}}u^2\,d\lambda
	= M\|u\|_{{\mathscr L}^2({\mathscr B})}^2.
	\]
	Applying this to $u=h_{x,\varepsilon}-p_d$ yields
	\[
	\|h_{x,\varepsilon}-p_d\|_{{\mathscr L}^2(\mu)}
	\le \sqrt{M}\,\|h_{x,\varepsilon}-p_d\|_{{\mathscr L}^2({\mathscr B})}
	\le C_4(\phi,f,k,n)\,d^{-k}\,\varepsilon^{-n/2-k}.
	\]	
	Now, $\Pi_d h_{x,\varepsilon}$ is the orthogonal projection of $h_{x,\varepsilon}$
	onto $V_d$ in ${\mathscr L}^2(\mu)$, hence the best ${\mathscr L}^2(\mu)$ approximation of
	$h_{x,\varepsilon}$ by degree-$d$ polynomials. Therefore
	\[
	\big\|h_{x,\varepsilon}-\Pi_d h_{x,\varepsilon}\big\|_{{\mathscr L}^2(\mu)}
	\le \big\|h_{x,\varepsilon}-p_d\big\|_{{\mathscr L}^2(\mu)}
	\le C_4(\phi,f,k,n)\,d^{-k}\,\varepsilon^{-n/2-k}.
	\]
	Rewriting this bound yields
	\[
	\big\|h_{x,\varepsilon}-\Pi_d h_{x,\varepsilon}\big\|_{{\mathscr L}^2(\mu)}
	\le C_4(\phi,f,k,n)\,\varepsilon^{-n/2}(\varepsilon d)^{-k},
	\]
	and squaring gives \eqref{eq:proj-L2mu} with
	$C(\phi,f,k,n):=C_4(\phi,f,k,n)^2$.
	
	Finally, since $\Pi_d$ is an orthogonal projection in ${\mathscr L}^2(\mu)$,
	\[
	\|h_{x,\varepsilon}\|_{{\mathscr L}^2(\mu)}^2
	= \|\Pi_d h_{x,\varepsilon}\|_{{\mathscr L}^2(\mu)}^2
	+ \|h_{x,\varepsilon}-\Pi_d h_{x,\varepsilon}\|_{{\mathscr L}^2(\mu)}^2,
	\]
	so the difference of squared norms is exactly the squared projection error,
	and \eqref{eq:proj-sqdiff} follows.
\end{proof}

\begin{remark}[Normalized projection error]\label{rmk:NPE}
	Recall that
	\[
	\|\phi_{x,\varepsilon}\|_{{\mathscr L}^2(\lambda)}^2
	= \int_\X \left(\frac{1}{\epsilon^n}\phi\left(\frac{x-y}{\epsilon}\right)\right)^2\,d\lambda(y)
	= \varepsilon^{-n}\|\phi\|_{{\mathscr L}^2(\R^n)}^2
	\]
	whenever $\mathscr B(z,\epsilon)\subseteq X$. Thus, when we consider the \emph{normalized} projection
	error appearing in~\eqref{eq:error-decomposition},
	\[
	\frac{
		\big|\|h_{x,\varepsilon}\|_{{\mathscr L}^2(\mu)}^2
		- \|\Pi_d h_{x,\varepsilon}\|_{{\mathscr L}^2(\mu)}^2\big|
	}{
		\|\phi_{x,\varepsilon}\|_{{\mathscr L}^2(\lambda)}^2
	}
	\;\le\;
	C''(\phi,f,k,n)\,(\varepsilon d)^{-2k},
	\]
	for some constant $C''$ independent of $x\in K$, $\varepsilon\leq \varepsilon_0$ and $d$. This is the
	form that enters directly in the density estimator error analysis.
\end{remark}

\subsubsection{Convergence rates with local support mollifiers}

Combining the approximation error (Lemma~\ref{lem:approx-error}) and the
projection error (Lemma~\ref{lem:proj-error}) with the error
decomposition~\eqref{eq:error-decomposition}, we obtain the following
quantitative density recovery result.

\begin{theorem}[Density recovery with explicit rate]
	\label{thm:rate-general}
	If Assumption~\ref{assumpt:Rn} and Assumption~\ref{ass:sobolev} hold with a density
	$f = 1/g \in \mathscr C^2(\X)$ and
	 $k$ is an integer with $1\le k\le s$ and $k>n/2$ as in
	Lemma~\ref{lem:proj-error} then for every compact set
	$K\subset\interior\X$ there exist constants $A,B>0$ and
	$\varepsilon_0>0$ such that for every $x\in K$, every $d\in\mathbb N$
	and every $0<\varepsilon<\varepsilon_0$,
	\begin{equation}
		\label{eq:rate-general}
		\bigl|\widehat{g}_{d,\varepsilon}(x) - g(x)\bigr|
		\;\le\;
		A\,(\varepsilon d)^{-2k} \;+\; B\,\varepsilon^2,
	\end{equation}
	where $\widehat{g}_{d,\varepsilon}$ is defined
	in~\eqref{eq:estimator}. In particular, choosing $\varepsilon
	= d^{-\frac{k}{k+1}}$ yields
	\[
	\bigl|\widehat{g}_{d,\varepsilon_d}(x) - g(x)\bigr|
	= O\bigl(d^{-\frac{2k}{k+1}}\bigr)
	\]
	uniformly on compact subsets of $\interior\X$.
\end{theorem}

\begin{proof}
	Fix a compact set $K\subset\interior\X$.
	For each $x\in K$, $d\in\mathbb N$ and $0<\varepsilon<\varepsilon_0$,
	the error decomposition~\eqref{eq:error-decomposition} gives
	\[
	\bigl|\widehat{g}_{d,\varepsilon}(x) - g(x)\bigr|
	\;\le\;
	e_{\mathrm{proj}}(x,d,\varepsilon)
	+ e_{\mathrm{approx}}(x,\varepsilon),
	\]
	where
	\[
	e_{\mathrm{proj}}(x,d,\varepsilon)
	:=
	\frac{
		\bigl|
		\| \Pi_d h_{x,\varepsilon}\|_{{\mathscr L}^2(\mu)}^2
		- \|h_{x,\varepsilon}\|_{{\mathscr L}^2(\mu)}^2
		\bigr|
	}{
		\|\phi_{x,\varepsilon}\|_{{\mathscr L}^2(\lambda)}^2
	},
	\qquad
	e_{\mathrm{approx}}(x,\varepsilon)
	:=
	\left|
	\frac{
		\|h_{x,\varepsilon}\|_{{\mathscr L}^2(\mu)}^2
	}{
		\|\phi_{x,\varepsilon}\|_{{\mathscr L}^2(\lambda)}^2
	}
	- g(x)
	\right|.
	\]
	
	\medskip\noindent
	\emph{Step 1: Approximation error.}
	By Lemma~\ref{lem:approx-error} (approximation error), since $f\in \mathscr C^2(\X)$ and
	$K\subset\interior\X$, there exist constants $B>0$ and $\varepsilon_0>0$
	such that, for all $x\in K$ and all $0<\varepsilon<\varepsilon_0$,
	\begin{equation}\label{eq:bound-approx}
		e_{\mathrm{approx}}(x,\varepsilon)
		\;\le\; B\,\varepsilon^2.
	\end{equation}
	The constants depend only on $f$, $\phi$ and $K$, and not on $d$.
	
	\medskip\noindent
	\emph{Step 2: Projection error.}
	By Lemma~\ref{lem:proj-error} (projection error), under
	Assumption~\ref{ass:sobolev} we have that there exists $\varepsilon_0$ such that
$\mathscr B(x,\varepsilon_0)
	\subset\interior X\subset{\mathscr B}$ for every $x\in K$ and such that for all $\epsilon<\varepsilon_0$,
	\begin{equation}\label{eq:bound-proj-L2}
		0\le
		\|h_{x,\varepsilon}\|_{{\mathscr L}^2(\mu)}^2
		- \|\Pi_d h_{x,\varepsilon}\|_{{\mathscr L}^2(\mu)}^2
		= \|h_{x,\varepsilon} - \Pi_d h_{x,\varepsilon}\|_{{\mathscr L}^2(\mu)}^2
		\;\le\;
		C_1(\phi,f,k,n)\,\varepsilon^{-n}(\varepsilon d)^{-2k}.
	\end{equation}
	
	On the other hand, by the direct computation of Remark~\ref{rmk:NPE}, the denominator satisfies	\[
	\|\phi_{x,\varepsilon}\|_{{\mathscr L}^2(\lambda)}^2
	= \varepsilon^{-n}\|\varphi\|_{{\mathscr L}^2(\R^n)}^2
	=: C_\varphi\,\varepsilon^{-n},
	\]
	independently of $x\in K$ and $d$ whenever $\varepsilon<\varepsilon_0$. 	
    Therefore
	\[
	e_{\mathrm{proj}}(x,d,\varepsilon)
	= \frac{
		\bigl|
		\|\Pi_d h_{x,\varepsilon}\|_{{\mathscr L}^2(\mu)}^2
		- \|h_{x,\varepsilon}\|_{{\mathscr L}^2(\mu)}^2
		\bigr|
	}{
		\|\phi_{x,\varepsilon}\|_{{\mathscr L}^2(\lambda)}^2
	}
	\le
	\frac{
		C_1(\varphi,f,k,n)\,\varepsilon^{-n}(\varepsilon d)^{-2k}
	}{
		C_\varphi\,\varepsilon^{-n}
	}
	=
	A\,(\varepsilon d)^{-2k},
	\]
	with $A := C_1(\varphi,f,k,n)/C_\varphi$ independent of $x$, $d$ and
	$\varepsilon$.

	Thus, for all $x\in K$, all $d\in\mathbb N$ and all
	$0<\varepsilon<\varepsilon_0$,
	\begin{equation}\label{eq:combined-bound}
		\bigl|\widehat{g}_{d,\varepsilon}(x) - g(x)\bigr|
		\;\le\;
		A\,(\varepsilon d)^{-2k} \;+\; B\,\varepsilon^2,
	\end{equation}
	which is exactly~\eqref{eq:rate-general}.
	
	\medskip\noindent
	\emph{Step 3: Choice of $\varepsilon_d$ and optimization.}
	Letting $\varepsilon_d := d^{-\gamma}$ for some $\gamma\in(0,1)$.
	and substituting $\varepsilon_d$ into~\eqref{eq:combined-bound} yields
	\[
	\bigl|\widehat{g}_{d,\varepsilon_d}(x) - g(x)\bigr|
	\le
	A\,(\varepsilon_d d)^{-2k} + B\,\varepsilon_d^2
	= A\,d^{-2k(1-\gamma)} + B\,d^{-2\gamma}.
	\]
	Hence there exists a constant $C>0$, independent of $x\in K$ and $d$,
	such that
	\[
	\bigl|\widehat{g}_{d,\varepsilon_d}(x) - g(x)\bigr|
	\le C\,d^{-\theta},
	\]
	where
	\[
	\theta = \min\bigl\{2k(1-\gamma),\,2\gamma\bigr\}.
	\]
	
	To maximize $\theta$ over $\gamma\in(0,1)$, we balance the two exponents:
	\[
	2k(1-\gamma) = 2\gamma
	\quad\Longleftrightarrow\quad
	k(1-\gamma) = \gamma
	\quad\Longleftrightarrow\quad
	k = \gamma(k+1)
	\quad\Longleftrightarrow\quad
	\gamma = \frac{k}{k+1}.
	\]
	With this choice,
	\[
	2\gamma = \frac{2k}{k+1},
	\qquad
	2k(1-\gamma) = 2k\Bigl(1-\frac{k}{k+1}\Bigr)
	= \frac{2k}{k+1},
	\]
	so both terms decay like $d^{-2k/(k+1)}$ and
	\[
	\theta = \frac{2k}{k+1}.
	\]
	This shows that
	\[
	\bigl|\widehat{g}_{d,\varepsilon_d}(x) - g(x)\bigr|
	= O\bigl(d^{-2k/(k+1)}\bigr),
	\]
	uniformly for $x\in K$. Since $K\subset\interior\X$ was arbitrary,
	the convergence is uniform on compact subsets of $\interior\X$.
\end{proof}

\begin{remark}
	The integer $k$ can be chosen as any integer satisfying $n/2<k\le s$, where
	$s$ is the Sobolev regularity exponent in Assumption~\ref{ass:sobolev}.
	Increasing $s$ allows one to choose larger $k$, which improves the rate
	$2k/(k+1)$ towards~$2$. This is consistent with classical spectral
	approximation results for polynomial approximation in Sobolev spaces;
	see, e.g.,~\cite{LiXu2014} and references therein.
\end{remark}

	\subsection{Density recovery with algebraic mollifiers on the sphere}
	\label{sec:sphere}

Let $S:=S^{n-1}\subset\R^n$ be the unit sphere and let $\lambda$ denote the rotation-invariant probability measure on $S$. In this section we develop quantitative estimates for density recovery on the sphere using algebraic mollifiers as in Example~\ref{example:sphere}. 

More precisely, for positive integers $k$, let us define the family of mollifiers,
\begin{equation} \label{eq:polynomial_mollifier}
g_k(t) := \left(\frac{t+1}{2}\right)^k,
\qquad
\phi_{x,k}(y) := g_k(\langle x,y\rangle),
\qquad x,y\in S.
\end{equation}

We assume throughout that the density $d\mu(y)=f(y)d\lambda(y)$ satisfies $0<m\le f\le M < \infty$ for given constants $m$ and $M$ and $y\in S$. In particular, it follows that the support of the unknown measure is the whole sphere and that the support location and density estimator MCD kernels coincide. We denote both with the symbol
$K_{d}^{\mu}(x,y)$ defined by
\[
	K_{d,k}^{\mu}(x,y)
	:= \ip{\Pi_d h_{x,k}}{\Pi_d h_{y,k}}_{{\mathscr L}^2(\mu)},
	\]
where $h_{x,k}(y) := \frac{\phi_{x,k}(y)}{f(y)}$ and the projection is taken onto the space $V_d$ of spherical polynomials of degree $\leq d$.
As previously, we consider the estimator
\[
\widehat{g}_d(x) := \frac{p_d^\mu(x)}{\|\phi_{x,k}\|_{{\mathscr L}^2(\lambda)}^2},
\qquad g(x):=\frac{1}{f(x)}.
\]
and study the projection and approximation errors independently. Note that the degree $k$ of the polynomial $g_k$ plays the role of the chosen mollifier resolution $\epsilon$ whereas $d$ measures the degree up to which we know the moments of the measure $\mu$. At the end of the section we will select $k$ as a function of $d$ in a way which leads to optimal error estimates depending only on $d$.
	
\subsubsection{Approximation error}
\label{subsec:approx-error-sphere}

For each fixed $x\in S$, let $\nu_{x,k}$ be the probability measure
on $S$ with density proportional to $\phi_{x,k}(y)^2$:
\[
d\nu_{x,k}(y)
:= \frac{\phi_{x,k}(y)^2}{\displaystyle\int_{S}\phi_{x,k}(z)^2\,d\lambda(z)}
\,d\lambda(y).
\]

Then the approximation error at $x$ is given by
\[
\frac{\displaystyle\int_{S} \phi_{x,k}(y)^2 g(y)\,d\lambda(y)}
{\displaystyle\int_{S} \phi_{x,k}(y)^2\,d\lambda(y)}
- g(x)
= \int_{S} \bigl(g(y)-g(x)\bigr)\,d\nu_{x,k}(y).
\]

We first treat the Lipschitz case and then the $\mathscr C^2$ case.

\begin{lemma}[Approximation error on the sphere, Lipschitz case]
	\label{lem:approx-sphere-Lip}
	Assume that $g=1/f$ is Lipschitz on $S$ with Lipschitz constant $K$
	with respect to the geodesic distance $d_{S}(x,y)$.
	Then for all $x\in S$ and all $k\in\mathbb N$,
	\[
	\left|
	\frac{\displaystyle\int_{S} \phi_{x,k}(y)^2 g(y)\,d\lambda(y)}
	{\displaystyle\int_{S} \phi_{x,k}(y)^2\,d\lambda(y)}
	- g(x)
	\right|
	\le \,
    \frac{K \pi}{\sqrt{2}} \sqrt{\frac{n-1}{2k+n-1}}\, = \,O(k^{-\frac{1}{2}}).
	\]
\end{lemma}

\begin{proof}
	Fix $x\in S$ and $k\in\mathbb N$. By Lipschitz continuity of $g$,
	\[
	|g(y)-g(x)|
	\le K\,d_{S}(x,y),
	\qquad \forall y\in S.
	\]
	Hence
	\begin{align*}
		\left|
		\frac{\displaystyle\int_{S} \phi_{x,k}(y)^2 g(y)\,d\lambda(y)}
		{\displaystyle\int_{S} \phi_{x,k}(y)^2\,d\lambda(y)}
		- g(x)
		\right|
		&= \left|\int_{S} \bigl(g(y)-g(x)\bigr)\,d\nu_{x,k}(y)\right| \\
		&\le K \int_{S} d_{S}(x,y)\,d\nu_{x,k}(y).
	\end{align*}
	Applying Jensen's inequality to the nonnegative random variable
	$Y\mapsto d_{S}(x,Y)$ under $\nu_{x,k}$ yields
	\[
	\int_{S} d_{S}(x,y)\,d\nu_{x,k}(y)
	\le \sqrt{\int_{S} d_{S}(x,y)^2\,d\nu_{x,k}(y)}.
	\]
	On the sphere, the geodesic distance can be expressed as
	$d_{S}(x,y) = \arccos\langle x,y\rangle$; moreover,
	for $t\in[-1,1]$ one has the inequality
	\[
	\arccos(t) \le \frac{\pi}{\sqrt{2}}\,\sqrt{1-t},
	\]
	see  Appendix \ref{sec:trigo} for a proof.
	Therefore
	\[
	d_{S}(x,y)^2
	\le \frac{\pi^2}{2}\,(1-\langle x,y\rangle),
	\]
	and we obtain
	\begin{align*}
		\int_{S} d_{S}(x,y)^2\,d\nu_{x,k}(y)
		&\le \frac{\pi^2}{2}
		\int_{S} (1-\langle x,y\rangle)\,d\nu_{x,k}(y) \\
		&= \frac{\pi^2}{2}\,
		\frac{\displaystyle\int_{S} (1-\langle x,y\rangle)\phi_{x,k}(y)^2\,d\lambda(y)}
		{\displaystyle\int_{S} \phi_{x,k}(y)^2\,d\lambda(y)}.
	\end{align*}
	Combining the above estimates, we arrive at
	\begin{equation}\label{eq:error-Lip-ratio}
		\left|
		\frac{\displaystyle\int_{S} \phi_{x,k}(y)^2 g(y)\,d\lambda(y)}
		{\displaystyle\int_{S} \phi_{x,k}(y)^2\,d\lambda(y)}
		- g(x)
		\right|
		\le
		\frac{K\pi}{\sqrt{2}}\,
		\sqrt{
			\frac{\displaystyle\int_{S} (1-\langle x,y\rangle)\phi_{x,k}(y)^2\,d\lambda(y)}
			{\displaystyle\int_{S} \phi_{x,k}(y)^2\,d\lambda(y)}
		}.
	\end{equation}
	
	We now reduce the ratio of integrals to a univariate expression using the
	Funk-Hecke formula, see    Appendix \ref{sec:funk-hecke}. Since $\phi_{x,k}(y)^2$ depends only on
	$t:=\langle x,y\rangle$, we can write
    \[
    R_k :=
	\frac{\displaystyle\int_{S} (1-\langle x,y\rangle)\phi_{x,k}(y)^2\,d\lambda(y)}
	{\displaystyle\int_{S} \phi_{x,k}(y)^2\,d\lambda(y)}
	=
	\frac{ c_{\alpha} \displaystyle\int_{-1}^1 (1-t)\,g_{k}(t)^2\,w(t)\,dt}
	{ c_{\alpha}\displaystyle\int_{-1}^1 g_{k}(t)^2\,w(t)\,dt}.
	\]
    
	where $\alpha = \frac{n-2}{2}$, $w(t) := (1-t^2)^{\alpha-1/2}$ is the Gegenbauer weight and
	$c_\alpha>0$ is a normalization constant (independent of $k$). The constant $c_\alpha$ cancels in the ratio. 

    We can use the change of variables \(r = \frac{1+t}{2}\) in order to compute those integrals. Then
    \[
    t = 2r-1, \quad dt = 2 dr.
    \]
    This implies that
    \[
    1-t = 2(1-r), \quad 1+t = 2r, \quad 1-t^2 = (1-t)(1+t) = 4r(1-r).
    \]

    Now, the ratio \(R_k\) can be computed as
    \[
    R_k 
    = \frac{\int_0^1 2(1-r) \, r^{2k} \, (4r(1-r))^{\alpha -\frac{1}{2}} dr}{\int_0^1 r^{2k} \, (4r(1-r))^{\alpha-\frac{1}{2}} \, 2 dr}
    = 2 \frac{\int_0^1 r^{2k+\alpha-\frac{1}{2}} (1-r)^{\alpha+\frac{1}{2}} dr}
    {\int_0^1 r^{2k+\alpha-\frac{1}{2}} (1-r)^{\alpha-\frac{1}{2}} dr}.
    \]

    Let us express this integral as Beta functions. Recall
    \[
    \int_0^1 r^{p-1} (1-r)^{q-1} dr = B(p,q) = \frac{\Gamma(p)\Gamma(q)}{\Gamma(p+q)},
    \]
    and that \(\Gamma\) satisfies
    \[
    \Gamma(z+1)=z\Gamma(z).
    \]
    So we have that
    \begin{align*}
    R_k &= 
    2 \frac{B(2k +\alpha +\frac{1}{2},\alpha +\frac{3}{2})}{B(2k+\alpha+ \frac{1}{2},\alpha+\frac{1}{2})}\\
    &= 2 \frac{\Gamma(2k+\alpha+\frac{1}{2})\Gamma(\alpha +\frac{3}{2})}{\Gamma(2k+2\alpha+2)} \frac{\Gamma(2k+2\alpha +1)}{\Gamma(2k+\alpha+\frac{1}{2})\Gamma(\alpha +\frac{1}{2})}\\
    &= \frac{ 2\alpha + 1}{2k+2\alpha+1}.        
    \end{align*}
    Finally, if we substitute \(\alpha\) by \(\frac{n-2}{2}\), we obtain
    \[
    R_k = \frac{n-1}{2k+n-1}. 
    \]
\end{proof}

We now study the case under $\mathscr C^2$ regularity for the density.

\begin{lemma}[Approximation error on the sphere, $\mathscr C^2$ case]
	\label{lem:approx-sphere-C2}
	Assume that $g=1/f$ is strictly positive and belongs to $\mathscr C^2(S)$.
	Let
	\[
	\Lambda := \sup_{x\in S, z} \left|
	\frac{z^\top \nabla^2 g(x)z}{z^\top z} \right| \in (0,\infty),
	\]
	where $\nabla^2 g(x)$ is the Riemannian Hessian of $g$ at $x$ on $S$.
	Then, for all $x\in S$ and all $k\in\mathbb N$,
	\[
	\left|
	\frac{\displaystyle\int_{S} \phi_{x,k}(y)^2 g(y)\,d\lambda(y)}
	{\displaystyle\int_{S} \phi_{x,k}(y)^2\,d\lambda(y)}
	- g(x)
	\right|
	\le \frac{\pi^2 \Lambda}{4}\, \frac{n-1}{2k+n-1}
    = O(k^{-1}).
	\]
\end{lemma}

\begin{proof}
Fix $x\in S$ and let $\exp_x:T_xS\to S$ be the Riemannian exponential map at $x$.
For any $v\in T_xS$, the Riemannian Taylor expansion of $g$ reads
\[
g(\exp_x(v))
=
g(x)
+
\nabla g(x)[v]
+
\int_0^1 (1-s)\,\nabla^2 g(\exp_x(sv))[v,v]\,ds,
\]
where $\nabla g(x)\in T_xS$ is the linear form (gradient) and $\nabla^2 g(x)$ is the symmetric bilinear form (Hessian). 
    
    By the definition of $\Lambda$, we obtain
\[
\left|
\int_0^1 (1-s)\,\nabla^2 g(\exp_x(sv))[v,v]\,ds
\right|
\le
\Lambda \|v\|^2 \int_0^1 (1-s)\,ds
=
\frac{\Lambda}{2}\|v\|^2.
\] 
Any point $y\in S$ can be written as $y=\exp_x(v)$ with
$v = \exp_x^{-1}(y)\in T_xS$ and $\|v\| = d_{S}(x,y)$. On
the unit sphere, one has the explicit formula
\[
v = \exp_x^{-1}(y)
= \arccos(\langle x,y\rangle)\,u_y,
\]
where $u_y\in T_xS$ is a unit tangent vector orthogonal to $x$.
Thus $\|v\|=d_{S}(x,y) = \arccos(\langle x,y\rangle)$, and
\[
g(y)
=
g(x)
+
\arccos(\langle x,y\rangle)\,\nabla g(x)[u_y]
+
R_x(y),
\]
where the remainder satisfies
\[
|R_x(y)|
\le
\frac{\Lambda}{2}\,\arccos(\langle x,y\rangle)^2.
\]

Integrating against $\nu_{x,k}$ as before, we obtain
\[
\frac{\displaystyle\int_{S} \phi_{x,k}(y)^2 g(y)\,d\lambda(y)}
{\displaystyle\int_{S} \phi_{x,k}(y)^2\,d\lambda(y)}
=
g(x) + T_1 + T_2,
\]
where
\begin{align*}
T_1
&:=
\int_{S}
\arccos(\langle x,y\rangle)\,\nabla g(x)[u_y]
\,d\nu_{x,k}(y), \\
T_2
&:=
\int_{S}
R_x(y)
\,d\nu_{x,k}(y).
\end{align*}

	Because $\phi_{x,k}^2$ depends only on the inner product $\langle x,y\rangle$,
	the measure $\nu_{x,k}$ is invariant under rotations around the axis through $x$.
	In particular, the distribution of the tangent direction $u_y$ is symmetric and
	has zero mean. This implies
	\[
	\int_{S} \arccos(\langle x,y\rangle)\,u_y\,d\nu_{x,k}(y) = 0,
	\]
	and hence $T_1=0$. Thus the linear term vanishes by symmetry, exactly as in the
	Euclidean radial-mollifier case.
    
	For the quadratic term, we get
	\[
	|T_2|
	\le \frac{\Lambda}{2}
	\int_{S} \arccos(\langle x,y\rangle)^2\,d\nu_{x,k}(y).
	\]
	
	As before, we use the inequality
	\[
	\arccos(t) \le \frac{\pi}{\sqrt{2}}\sqrt{1-t},
	\qquad t\in[-1,1],
	\]
	of Appendix \ref{sec:trigo}.
	Therefore,
	\begin{equation*}
		|T_2|
		\le \frac{\pi^2 \Lambda}{4}\,
		\int_{S} (1-\langle x,y\rangle)\,d\nu_{x,k}(y).
	\end{equation*}

    As in the Lipschitz case, by the Funk-Hecke reduction of Appendix \ref{sec:funk-hecke},
	\[
		\int_{S} (1-\langle x,y\rangle)\,d\nu_{x,k}(y)
	= \frac{\int_{-1}^1 (1-t) \, g_k(t)^2 \,w(t)\,dt}
	{\int_{-1}^1 \, g_k(t)^2 \, w(t)\,dt}
	= \frac{n-1}{2k+n-1},
	\]
	by the same argument as above. This concludes the proof.
\end{proof}

\subsubsection{Projection error}

The projection error on the sphere involves approximating the localized
functions
\[
h_{x,k}(y) = \frac{\phi_{x,k}(y)}{f(y)},\qquad x,y\in S,
\]
by spherical polynomials of degree at most $d$ in ${\mathscr L}^2(\mu)$.
We control this error using spectral approximation results for spherical
harmonics.

Let ${\mathscr H}^s(S)$ denote the Sobolev space on the sphere of order
$s>0$, defined via the Laplace--Beltrami operator, see e.g.
\cite{DaiXu2013}.
For $s>(n-1)/2$, the space ${\mathscr H}^s(S)$ embeds continuously into
${\mathscr C}^0(S)$.

A standard Jackson-type inequality for spherical harmonics, see, e.g., \cite[Cor.~4.5.6]{DaiXu2013}, states that for any $s>0$
there exists a constant $C_s>0$ such that for every
$h\in{\mathscr H}^s(S)$ and every integer $d\ge1$,
\begin{equation}
	\label{eq:jackson-sphere}
	e_d(h)_{{\mathscr L}^2(\lambda)}
	:= \inf_{p\in{\mathscr V}_d}\|h-p\|_{{\mathscr L}^2(\lambda)}
	\le C_s\,d^{-s}\,\|h\|_{{\mathscr H}^s(S)},
\end{equation}
where ${\mathscr V}_d$ denotes the space of spherical polynomials of degree at
most $d$.

In our setting $h_{x,k} = \phi_{x,k}\,g$ with $g=1/f$ and $\phi_{x,k}$ a zonal
polynomial of degree $k$.

\begin{lemma}[Normalized projection error on the sphere]
\label{lem:proj-error-sphere-normalized}
Assume that $f$ is a strictly positive density on $S$ with respect to
$\lambda$, with $0<m\le f\le M<\infty$. Let \(\phi_{x,k}\) be the family of algebraic mollifiers as in \eqref{eq:polynomial_mollifier}.
Fix $s\in\{1,2\}$. Assume that \(g\in\mathscr H^s(S)\) and \(\partial^\beta g\in L^\infty(S)\) for every \(|\beta|\le s\).

Then:
\begin{enumerate}
\item If \(s=1\), the normalized projection error satisfies
\[
\frac{\big|\|h_{x,k}\|_{{\mathscr L}^2(\mu)}^2 - \|\Pi_d h_{x,k}\|_{{\mathscr L}^2(\mu)}^2\big|}
{\|\phi_{x,k}\|_{{\mathscr L}^2(\lambda)}^2}
= O \left( \frac{k}{d^2} \right).
\]
\item If \(s=2\), the normalized projection error satisfies
\[
\frac{\big|\|h_{x,k}\|_{{\mathscr L}^2(\mu)}^2 - \|\Pi_d h_{x,k}\|_{{\mathscr L}^2(\mu)}^2\big|}
{\|\phi_{x,k}\|_{{\mathscr L}^2(\lambda)}^2}
= O \left(\frac{k^2}{d^4} \right).
\]
\end{enumerate}
\end{lemma}

\begin{proof}
    First, because of the Pythagorean Theorem, we know that:
	\begin{equation}
		0\le
		\big\|h_{x,k}\big\|_{{\mathscr L}^2(\mu)}^2
		- \big\|\Pi_d h_{x,k}\big\|_{{\mathscr L}^2(\mu)}^2
		= \big\|h_{x,k} - \Pi_d h_{x,k}\big\|_{{\mathscr L}^2(\mu)}^2.
	\end{equation}
    
	We first work in ${\mathscr L}^2(\lambda)$. Since $0<m\le f\le M<\infty$ on
	$S$, the norms $\|\cdot\|_{{\mathscr L}^2(\mu)}$ and
	$\|\cdot\|_{{\mathscr L}^2(\lambda)}$ are equivalent:
	\[
	m^{1/2}\|u\|_{{\mathscr L}^2(\lambda)}
	\le \|u\|_{{\mathscr L}^2(\mu)}
	\le M^{1/2}\|u\|_{{\mathscr L}^2(\lambda)},\qquad \forall u.
	\]
	In particular,
	\[
	\big\|h_{x,k} - \Pi_d h_{x,k}\big\|_{{\mathscr L}^2(\mu)}
	\le M^{1/2}\,\big\|h_{x,k} - \Pi_d h_{x,k}\big\|_{{\mathscr L}^2(\lambda)}.
	\]
	Using that $\Pi_d h_{x,k}\in{\mathscr V}_d$, we have
	\[
	\big\|h_{x,k} - \Pi_d h_{x,k}\big\|_{{\mathscr L}^2(\lambda)}
	\leq \inf_{p \in \mathscr V_d} \norm{h_{x,k}-p}_{{\mathscr L}^2(\lambda)}.
	\]

    By assumption \(g\in\mathscr H^s(S)\) and \(\partial^{\beta} g\) is bounded for every \(|\beta|\leq s\). The Sobolev space defined via the Laplace--Beltrami operator admits an equivalent characterization in terms of angular derivatives $D_{i,j}$, see \cite[Chap.~4, Sec.~1.8]{DaiXu2013}. Since these operators are linear combinations of Euclidean derivatives, the Sobolev norm is equivalent to a norm involving Euclidean derivatives restricted to the sphere.
    Now apply the Jackson inequality~\eqref{eq:jackson-sphere} to $h=h_{x,k}$:
	\[
	\inf_{p \in \mathscr V_d} \norm{h_{x,k}-p}_{{\mathscr L}^2(\lambda)}
	\le C_s\,d^{-s}\,\|h_{x,k}\|_{{\mathscr H}^s(S)}
    \leq C'_s\,d^{-s}\, \sqrt{\sum_{|\alpha|\le s} \|\partial^\alpha h_{x,k}\|_{{\mathscr L}^2({\lambda})}^2},
	\]
    for some constants \(C_s, C'_s >0\).
        
    Using the Leibniz rule, we can bound \( \|\partial^\beta h_{x,k}\|_{{\mathscr L}^2({\lambda})}\) by \(\|\partial^\beta \phi_{x,k}\|_{{\mathscr L}^2({\lambda})}\) times a constant that depends on the bound of \( \partial^{\beta} g\).
    Then, there exists a constant \(C>0\) depending only on \(s\) and \(g\) such that
    \[
    \frac{\|h_{x,k} -\Pi_d h_{x,k}\|_{{\mathscr L}^2(\lambda)}^2}
        {\|\phi_{x,k}\|_{{\mathscr L}^2(\lambda)}^2}
    \leq
    C\,d^{-2s}
    \frac{\sum_{|\alpha|\le s} \|\partial^\alpha \phi_{x,k}\|_{{\mathscr L}^2({\lambda})}^2}
    {\|\phi_{x,k}\|_{{\mathscr L}^2(\lambda)}^2}.
    \]

    \textbf{Case $\boldsymbol{s=1}$.} 
    In this case, the previous bound reduces to
    \[
    \frac{\|h_{x,k} -\Pi_d h_{x,k}\|_{{\mathscr L}^2(\lambda)}^2}
        {\|\phi_{x,k}\|_{{\mathscr L}^2(\lambda)}^2}
    \leq
    C\,d^{-2}
    \left( 1 +
    \frac{ \|\nabla \phi_{x,k}\|_{{\mathscr L}^2({\lambda})}^2}
    {\|\phi_{x,k}\|_{{\mathscr L}^2(\lambda)}^2} \right).
    \]

    We want to study the behavior of
    \[
    R_1(k) := \frac{ \|\nabla \phi_{x,k}\|_{{\mathscr L}^2({\lambda})}^2}
    {\|\phi_{x,k}\|_{{\mathscr L}^2(\lambda)}^2}
    \]
    Observe that
    \[
    |\nabla_{S^{n-1}}\phi_{x,k}(y)|^2=(1-t^2)|g_k'(t)|^2,
    \]
    and 
    \[
    g_k'(t)=\frac{k}{2}\Big(\frac{t+1}{2}\Big)^{k-1}.
    \]
    and for a zonal function,
    
    Thus, again by the Funk-Hecke reduction of Appendix \ref{sec:funk-hecke} and the change of variables \(r=(t+1)/2\),
    \[
    R_1 (k) 
    =\frac{c_{\alpha}\frac{k^2}{4}\int_{-1}^1
    \Big(\frac{t+1}{2}\Big)^{2k-2}(1-t^2)^{\alpha+\frac{1}{2}}\,dt.}
    {c_{\alpha} \int_{-1}^1 \Big(\frac{t+1}{2}\Big)^{2k}
    (1-t^2)^{\alpha-\frac{1}{2}}\,dt.}
    =\frac{k^2\, \int_0^1 r^{\,2k + \alpha  -\frac{3}{2}} (1-r)^{\alpha+\frac{1}{2}} \,dr.}
    {\int_0^1 r^{2k+\alpha-\frac{1}{2}} (1-r)^{\alpha-\frac{1}{2}}\,dr}
    \]
    If we express this in term of Beta function,
    \[
    R_1(k) = k^2 \frac{B(\,2k + \alpha  -\frac{1}{2},\alpha+\frac{3}{2})}
    {B(2k+\alpha +\frac{1}{2},\alpha +\frac{1}{2})}
    = k^2
    \frac{\Gamma\left( 2k + \alpha - \frac{1}{2} \right) \Gamma\left( \alpha + \frac{3}{2} \right)}
    {\Gamma(2k + 2\alpha + 1)}
    \frac{\Gamma(2k + 2\alpha + 1)}
    {\Gamma\left( 2k + \alpha + \frac{1}{2} \right) \Gamma\left( \alpha + \frac{1}{2} \right)}
    \]
    Recall that \( \Gamma(z+1)=z\Gamma(z)\). Then
    \[
    R_1(k) = k^2 \frac{\alpha+\frac{1}{2}}{2k+\alpha-\frac{1}{2}}
    = k^2 \frac{n-1}{4k+n-3},
    \]
    and the first statement holds.

    \textbf{Case $\boldsymbol{s=2}$.}
    Assume now that $g\in\mathscr H^2(S)$ and that all derivatives
    of order $\le2$ are bounded. Applying Jackson's inequality with $s=2$,
    \[
    \frac{\|h_{x,k} -\Pi_d h_{x,k}\|_{{\mathscr L}^2(\lambda)}^2}{\|\phi_{x,k}\|_{{\mathscr L}^2(\lambda)}^2}
    \le
    C\,d^{-4} \left( 1 +R_1(k) +R_2(k)\right),
    \]
    where
    \[
    R_2(k) := \frac{\|D^2 \phi_{x,k}\|_{{\mathscr L}^2(\lambda)}^2}
    {\|\phi_{x,k}\|_{{\mathscr L}^2(\lambda)}^2}.
    \]
    
    Again, since $\phi_{x,k}$ is zonal with respect to $x$, direct computation gives
    \[
    |D^2_{S^{n-1}}\phi_{x,k}(y)|^2
    =
    (1-t^2)^2 |g_k''(t)|^2 + (n-1)(1-t^2)|g_k'(t)|^2,
    \]
    where \(g_k''(t)=\frac{k(k-1)}{4}\Big(\frac{t+1}{2}\Big)^{k-2}\).
    
    Therefore, by the Funk--Hecke reduction,
    \[
    \|D^2 \phi_{x,k}\|_{{\mathscr L}^2(\lambda)}^2
    =
    c_\alpha
    \int_{-1}^1
    \Big[
    (1-t^2)^2 |g_k''(t)|^2
    +
    (n-1)(1-t^2)|g_k'(t)|^2
    \Big]
    (1-t^2)^{\alpha-\frac12}\,dt.
    \]
    
    Splitting the two contributions,
    \[
    \|D^2 \phi_{x,k}\|_{{\mathscr L}^2(\lambda)}^2
    =
    A_k + B_k,
    \]
    where
    \[
    A_k
    =
    c_\alpha \frac{k^2(k-1)^2}{16}
    \int_{-1}^1
    \Big(\frac{t+1}{2}\Big)^{2k-4}
    (1-t^2)^{\alpha+\frac32}\,dt,
    \]
    and
    \[
    B_k
    =
    c_\alpha (n-1)\frac{k^2}{4}
    \int_{-1}^1
    \Big(\frac{t+1}{2}\Big)^{2k-2}
    (1-t^2)^{\alpha+\frac12}\,dt.
    \]
    
    The term $B_k$ is of the same type as in $R_1(k)$ and therefore contributes $O(k)$ after normalization.
    
    Using again the change of variables $r=(t+1)/2$, we obtain
    \[
    A_k = k^2(k-1)^2
    \frac{ \int_0^1 r^{\,2k+\alpha-\frac52} (1-r)^{\alpha+\frac32}\,dr}
    {\int_0^1 r^{\,2k+\alpha-\frac12} (1-r)^{\alpha-\frac12}\,dr}
    \]
    
    Repeating the same arguments with the Beta function and the Stirling's formula, it holds
    \[
    A_k = O(k^2).
    \]
    
    Since \(R_1(k)=O(k)\) and \(B_k = O(k)\), the dominant contribution is $A_k$ and this completes the proof.
\end{proof}

\subsubsection{Convergence rates}

We can now combine the approximation and projection error bounds into a
convergence Theorem for density recovery on the sphere.

\begin{theorem}[Improved density recovery on the sphere]
	\label{thm:rate-sphere}
	Let $S\subset\R^n$ be the unit sphere and $\lambda$ the normalized
	surface measure. Let $\mu$ be a probability measure on $S$ with
	strictly positive density $f$ with respect to $\lambda$. Assume that:
	\begin{itemize}
		\item[(i)] $0<m\le f\le M<\infty$ on $S$;
		\item[(ii)] the algebraic mollifiers $\phi_{x,d}$ are constructed as in  \eqref{eq:polynomial_mollifier}, choosing \(k = \lfloor d^{4/3} \rfloor.\)
	\end{itemize}

    Then, the estimator \(\widehat g_d(x)\) satisfies:
    \begin{itemize}
        \item if \(g\in\mathscr C^1(S)\),
        \(\big|\widehat g_d(x)-g(x)\big| = O\!\left(d^{-2/3}\right)\);
        \item if \(g\in\mathscr C^2(S)\),
        \(\big|\widehat g_d(x)-g(x)\big| = O\!\left(d^{-4/3}\right)\);
    \end{itemize}
    uniformly in \(x\in S\).
    The implicit constants depend only on \(n\), \(m,M\), and the Sobolev/bounded-derivative constants of \(g\).
\end{theorem}

\begin{proof}
Recall \(h_{x,k}=\phi_{x,k}/f\), and let \(\Pi_d\) denote the
orthogonal projection onto spherical polynomials of degree \(\le d\)
in \({\mathscr L}^2(\mu)\).
The error decomposition is
\[
\big|\widehat{g}_d(x) - g(x)\big|
\le \underbrace{
    \frac{
        \big|\|h_{x,k}\|_{{\mathscr L}^2(\mu)}^2
        - \|\Pi_d h_{x,k}\|_{{\mathscr L}^2(\mu)}^2\big|
    }{
        \|\phi_{x,k}\|_{{\mathscr L}^2(\lambda)}^2
    }
}_{\text{projection error}}
+ \underbrace{
    \left|
    \frac{\|h_{x,k}\|_{{\mathscr L}^2(\mu)}^2}{\|\phi_{x,k}\|_{{\mathscr L}^2(\lambda)}^2}
    - g(x)
    \right|
}_{\text{approximation error}}.
\]

For the case of regularity \(s=1\), from Lemma~\ref{lem:approx-sphere-Lip} we know that the approximation error is \( O \left(k^{-1/2}\right) \), and from Lemma~\ref{lem:proj-error-sphere-normalized} that the projection error is \( O\left( k d^{-2} \right) \). 
Setting \(k=d^\gamma\) and equating exponents, we obtain
\[
-\frac{\gamma}{2}=\gamma-2
\quad\Longrightarrow\quad
\gamma=\frac{4}{3}.
\]
Thus both terms are \(O(d^{-2/3})\).

In the case \(g\in\mathscr C^2(S)\), applying Lemmas~\ref{lem:approx-sphere-C2} and~\ref{lem:proj-error-sphere-normalized}, the approximation error is \(O(k^{-1})\) and the projection error is \(O(k^2 d^{-4})\).
Setting again $k=d^\gamma$ and equating exponents, we obtain
\[
-\gamma = 2\gamma - 4
\quad\Longrightarrow\quad
\gamma=\frac{4}{3}.
\]
Thus both terms are \(O(d^{-4/3})\).
\end{proof}

\begin{remark}
The regularity assumption in the first part can be weakened.
It is sufficient to assume that $g\in\mathscr H^1(S)$ and that $g$
is Lipschitz, which are the hypotheses required in
Lemmas~\ref{lem:approx-sphere-Lip} and
\ref{lem:proj-error-sphere-normalized}.
\end{remark}

    \section{Numerical experiments}
\label{sec:numerics}

    In this section we report some numerical experiments illustrating the performance of the density estimator based on the mollified Christoffel--Darboux (MCD) kernel on the sphere. 
    The Python code used to produce these examples and additional numerical details are available via GitHub at \url{https://github.com/LeandroBentancur/mollifiedcdkernel}. 

    First, we build an orthonormal basis of spherical harmonics \(b = \{ b_1, \ldots, b_N \}\) and the family of algebraic mollifiers constructed from Gegenbauer polynomials introduced in Subsection~\ref{subsec:approx-error-sphere}. Then, Lemma \ref{lem: CD_kernel_basics} gives us an explicit representation 
    \[
    DMCD_{d,\epsilon}^{\mu}(x,x) = [\Pi_d(\phi_{x,d})]_b^\top M^{-1}(\mu) [\Pi_d(\phi_{x,d})]_b,
    \]
    where \([\Pi_d(\phi_{x,d})]_b\) are the coefficients of \(\Pi_d(\phi_{x,d})\) in the basis \(b\). As \(b\) is an orthonormal basis of spherical harmonics, we can obtain these coefficients using the Funk-Hecke formula (Theorem \ref{thm:Funk-Hecke}) just as a projection of \(g_d\) to \(V_d\). All integrals are approximated using quadrature rules on \(S^2\).

    The computational bottleneck of the method is solving the linear systems associated with the moment matrix, whose dimension is of order \(\binom{n+d}{n,d}\). In our implementation this is handled via a Cholesky factorization. In addition, for very concentrated mollifiers, numerical instabilities may arise.
    
    The von Mises-Fisher distribution in the sphere with mean direction \(\mu\) and concentration parameter \(\kappa\) is defined by the density function \(f(x;\mu,\kappa) = C(\kappa) \exp(\kappa \mu^\top x)\), where \(\kappa >0\), \( \lVert \mu \rVert = 1\) and \(C(\kappa)\) is a normalization constant. 
    We denote by \(f_{\kappa}\) the equally weighted mixture of three von Mises-Fisher distributions with means on the three canonical basis vectors of \(\mathbb R^3\) and common concentration parameter \(\kappa\). Specifically, we analyze the case where $\kappa = 3$ (denoted as $f_3$) to evaluate the behavior of the MCD kernel density estimator over $S^2$ using the family of algebraic mollifiers.
        
    \begin{figure}[htbp]
    \centering
    \includegraphics[width=0.65\textwidth]{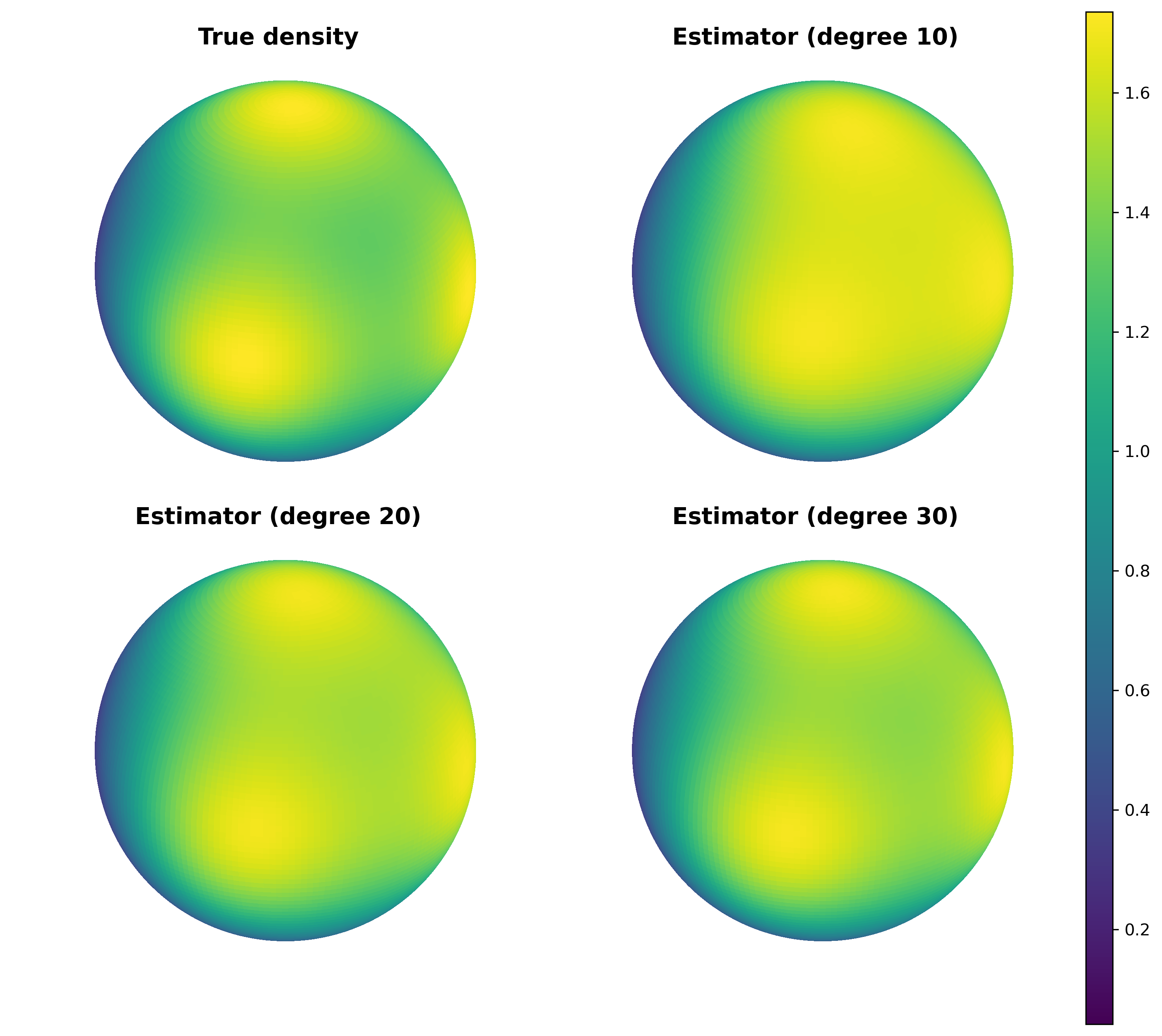}
    \caption{Plot of the $f_3$ density alongside its approximations using the MCD kernel density estimator for degrees $10$, $20$, and $30$.}
    \label{fig:plot_comparison}
    \end{figure}
    
    Figure~\ref{fig:plot_comparison} illustrates the true $f_3$ density and its MCD kernel approximations for various degrees. As shown in Figure~\ref{fig:chart_errors}, the approximation error exhibits a decay of $\mathcal{O}(d^{-4/3})$, which dominates the total error. These experiments demonstrate that the MCD kernel reproduces the theoretically predicted error decomposition in a nontrivial example on $S^2$.
    
    \begin{figure}[htbp]
    \centering
    \includegraphics[width=0.6\textwidth]{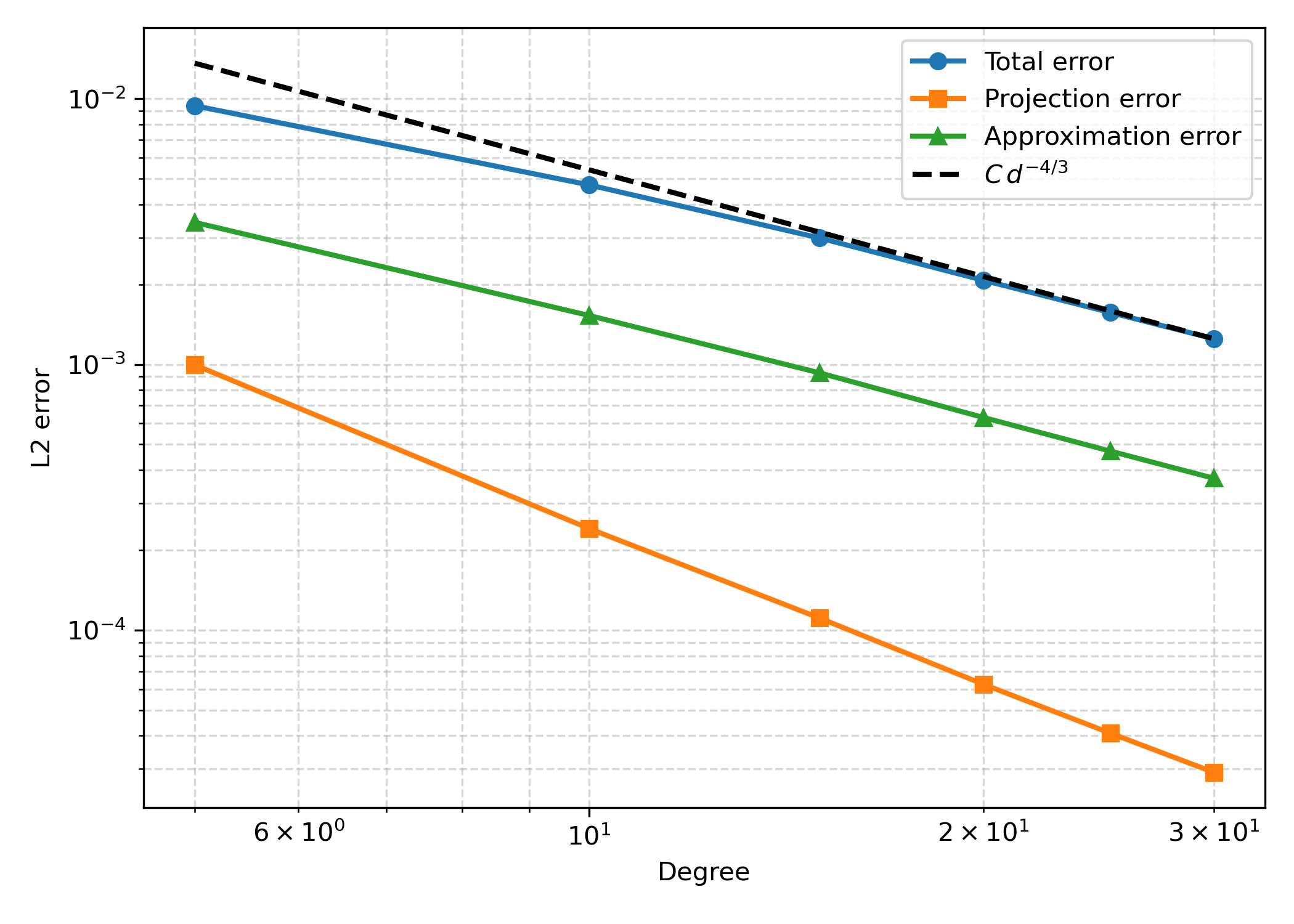}
    \caption{${\mathscr L}^2$ errors for the recovery of the $f_3$ density as a function of the degree $d$.}
    \label{fig:chart_errors}
    \end{figure}
    	
\appendix

\section{Trigonometric inequality}\label{sec:trigo}

\begin{lemma}
	\label{lem:arccos-bound}
	For every $\theta\in[0,\pi]$ one has
	\begin{equation}
		\label{eq:cos-quadratic}
		1 - \cos\theta \;\ge\; \frac{2}{\pi^2}\,\theta^2.
	\end{equation}
	Equivalently, for every $t\in[-1,1]$,
	\begin{equation}
		\label{eq:arccos-root}
		\arccos(t) \;\le\; \frac{\pi}{\sqrt{2}}\,\sqrt{1-t}.
	\end{equation}
\end{lemma}

\begin{proof}
	Define
	\[
	h(\theta) := \frac{1 - \cos\theta}{\theta^2},\qquad \theta\in(0,\pi].
	\]
	If we can show that $h$ is decreasing on $(0,\pi]$, then
	\[
	h(\theta) \;\ge\; h(\pi)
	= \frac{1-\cos\pi}{\pi^2}
	= \frac{2}{\pi^2},
	\]
	which is exactly \eqref{eq:cos-quadratic}.
	
	Differentiating,
	\[
	h'(\theta)
	= \frac{\theta^2\sin\theta - 2\theta(1-\cos\theta)}{\theta^4}
	= \frac{k(\theta)}{\theta^3},
	\]
	where
	\[
	k(\theta) := \theta\sin\theta - 2(1-\cos\theta).
	\]
	Thus $\operatorname{sign}(h'(\theta)) = \operatorname{sign}(k(\theta))$ for $\theta>0$.
	Differentiate $k$:
	\[
	k'(\theta) = \theta\cos\theta - \sin\theta,\qquad
	k''(\theta) = -\theta\sin\theta.
	\]
	For $\theta\in[0,\pi]$ we have $\sin\theta\ge0$, hence $k''(\theta)\le0$, so $k'$ is
	decreasing on $[0,\pi]$. Using the Taylor expansions $\sin\theta=\theta+O(\theta^3)$,
	$\cos\theta=1+O(\theta^2)$ as $\theta\to0$, one checks that $k'(0)=0$.
	Since $k'$ is decreasing and $k'(0)=0$, it follows that $k'(\theta)\le0$ for
	$\theta\in[0,\pi]$, so $k$ is decreasing on $[0,\pi]$. Again using Taylor expansion,
	$k(0)=0$, hence $k(\theta)\le0$ for all $\theta\in[0,\pi]$. Therefore $h'(\theta)\le0$
	on $(0,\pi]$, so $h$ is decreasing and \eqref{eq:cos-quadratic} holds.
	
	To obtain \eqref{eq:arccos-root}, let $t\in[-1,1]$ and set $\theta:=\arccos(t)\in[0,\pi]$.
	Then $t=\cos\theta$ and $1-t = 1-\cos\theta$, so \eqref{eq:cos-quadratic} gives
	\[
	1-t = 1-\cos\theta \;\ge\; \frac{2}{\pi^2}\,\theta^2,
	\]
	that is
	\[
	\theta \;\le\; \frac{\pi}{\sqrt{2}}\sqrt{1-t}.
	\]
	Since $\theta=\arccos(t)$, this is exactly \eqref{eq:arccos-root}.
\end{proof}

\section{The Funk-Hecke formula}\label{sec:funk-hecke}

In this appendix we recall the Funk-Hecke formula, which allows one to
reduce integrals of zonal functions against spherical harmonics to
one-dimensional integrals on $[-1,1]$.

Let $S:=S^{n-1}\subset\R^n$ denote the unit sphere,
and let $\lambda$ be the normalized surface measure on $S$.
For each integer $\ell\ge0$, let $\mathcal H_\ell$ denote the space of
spherical harmonics of degree $\ell$, i.e., the restrictions to $S$
of homogeneous harmonic polynomials of degree $\ell$ in $\R^n$.
It is well known that
\[
\mathscr {\mathscr L}^2(S,\lambda)
= \bigoplus_{\ell=0}^\infty \mathcal H_\ell,
\]
and that each $\mathcal H_\ell$ is finite-dimensional. We fix once and for all
an orthonormal basis $\{Y_{\ell,k}\}_{k=1}^{N_\ell}$ of $\mathcal H_\ell$ in
${\mathscr L}^2(S,\lambda)$.

A function $K:S\times S\to\R$ is called \emph{zonal}
(or rotationally invariant) if it depends only on the inner product:
\[
K(x,y) = F(\langle x,y\rangle),
\qquad x,y\in S,
\]
for some function $F:[-1,1]\to\R$. In this case we will also write
$K_F(x,y)$ for the kernel associated with $F$.

Let $\alpha := (n-1)/2$ and denote by $C_\ell^{(\alpha)}$ the Gegenbauer
polynomial of degree $\ell$. Let $w_\alpha(t) := (1-t^2)^{\alpha-1/2}$ be the
associated weight function on $[-1,1]$.

\begin{theorem}[Funk-Hecke formula]
	\label{thm:Funk-Hecke}
	Let $F:[-1,1]\to\mathbb C$ be such that
	\[
	\int_{-1}^1 |F(t)|\,w_\alpha(t)\,dt < \infty.
	\]
	Fix $x\in S$ and define the zonal kernel
	\[
	K_F(x,y) := F(\langle x,y\rangle),\qquad y\in S.
	\]
	Then for every spherical harmonic $Y_{\ell,k}\in\mathcal H_\ell$ one has
	\[
	\int_{S} K_F(x,y)\,Y_{\ell,k}(y)\,d\lambda(y)
	= \lambda_\ell(F)\,Y_{\ell,k}(x),
	\]
	where the scalar coefficient $\lambda_\ell(F)$ depends only on $\ell$ and $F$
	and is given explicitly by
	\[
		\lambda_\ell(F)
		= \frac{\omega_{n-2}}{\omega_{n-1}C_\ell^{(\alpha)}(1)}
		\int_{-1}^1 F(t)\,C_\ell^{(\alpha)}(t)\,w_\alpha(t)\,dt.
	\]
	Here $\omega_{n-1}$ and $\omega_{n-2}$ denotes the surface area of $S^{n-1}$ and $S^{n-2}$, respectively.
	In particular, the integral on the left-hand side is a multiple of
	$Y_{\ell,k}(x)$, and the multiple is the same for all $k=1,\dots,N_\ell$.
\end{theorem}

In the special case $\ell=0$, $C_0^{(\alpha)}(t)\equiv 1$ and $Y_{0,1}$ is
constant, so the formula reduces to
\[
	\int_{S} F(\langle x,y\rangle)\,d\lambda(y)
	= \frac{\omega_{n-2}}{\omega_{n-1}}
	\int_{-1}^1 F(t)\,w_\alpha(t)\,dt,
\]
which shows that such zonal integrals are independent of $x$ and can be reduced to a univariate integral on $[-1,1]$.

For more details, we refer to standard texts on spherical harmonics, e.g. \cite{Muller1966}.

\bibliographystyle{plain} 

\end{document}